\documentclass{amsart}

\usepackage{amssymb}
\usepackage{amsrefs}
\usepackage{amsmath}
\usepackage{amsfonts}
\usepackage{amsthm}
\usepackage{array}
\usepackage{bm}
\usepackage{enumitem}
\usepackage{xcolor}
\usepackage{eucal}

\usepackage{tikz}
\usetikzlibrary{cd}

\usepackage{hyperref}

\hypersetup{
	colorlinks,
	linkcolor={red!50!black},
	citecolor={blue!50!black},
	urlcolor={blue!80!black}
}

\newtheorem{thm}{Theorem}[section]
\newtheorem*{thm*}{Theorem}
\newtheorem{ithm}{Theorem}
\newtheorem{prop}[thm]{Proposition}

\newtheorem{cor}[thm]{Corollary}

\theoremstyle{definition}
\newtheorem{definition}[thm]{Definition}

\theoremstyle{remark}

\newtheorem{remark}[thm]{Remark}

\numberwithin{equation}{section}

\newcommand{\R}{\mathbb{R}}
\newcommand{\C}{\mathcal{C}}
\newcommand{\SO}{\mathrm{SO}}
\renewcommand{\O}{\mathrm{O}}
\newcommand{\SU}{\mathrm{SU}}
\newcommand{\Sg}{\mathrm{S}}

\newcommand{\G}{\mathrm{G}}
\newcommand{\Ber}{B}
\newcommand{\V}{\mathcal{V}} 
\newcommand{\Hc}{\mathcal{H}}
\newcommand{\ico}{\mathrm{Ico}}
\newcommand{\oct}{\mathrm{Oct}}
\newcommand{\tet}{\mathrm{Tet}}
\newcommand{\gra}{\mathrm{Gr}^+_2 \!\left( T S^4 \right)}

\renewcommand{\Re}{\operatorname{Re}}
\renewcommand{\Im}{\operatorname{Im}}

\begin{document}

\title{Associative submanifolds of the Berger space}

\author{Gavin Ball \and Jesse Madnick}

\newcommand{\Addresses}
{{  \bigskip
		%   \footnotesize
		\textsc{Universit\'{e} du Qu\'{e}bec \`{a} Montr\'{e}al} \par\nopagebreak
		\textsc{D\'{e}partement de math\'{e}matiques}\par\nopagebreak
		\textsc{Case postale 8888, succursale centre-ville, Montr\'{e}al (Qu\'{e}bec), H3C 3P8, Canada}\par\nopagebreak
		\textit{E-mail address}: \texttt{gavin.ball@cirget.ca} \\
		
		\bigskip
		\textsc{McMaster University} \par\nopagebreak
		\textsc{Department of Mathematics \& Statistics}\par\nopagebreak
		\textsc{Hamilton, ON, Canada, L8S 4K1}\par\nopagebreak
		\textit{E-mail address}: \texttt{madnickj@mcmaster.ca}
		
}}
%\urladdr{} % Delete if not wanted.

\begin{abstract}
We study associative submanifolds of the Berger space $\SO(5)/\SO(3)$ endowed with its homogeneous nearly-parallel $\G_2$-structure. We focus on two geometrically interesting classes: the ruled associatives, and the associatives with special Gauss map.

We show that the associative submanifolds ruled by a certain special type of geodesic are in correspondence with pseudo-holomorphic curves in $\gra.$ Using this correspondence, together with a theorem of Bryant on superminimal surfaces in $S^4,$ we prove the existence of infinitely many topological types of compact immersed associative 3-folds in $\SO(5)/\SO(3)$.

An associative submanifold of the Berger space is said to have special Gauss map if its tangent spaces have non-trivial $\SO(3)$-stabiliser. We classify the associative submanifolds with special Gauss map in the cases where the stabiliser contains an element of order greater than 2. In particular, we find several homogeneous examples of this type.

%We prove the existence of infinitely many topological types of compact immersed associative 3-folds in the Berger space $\Ber$.  Our examples are circle bundles over genus $g$ surfaces for any $g \geq 0$, and are obtained from a correspondence with pseudo-holomorphic curves in $\SO(5)/\mathrm{T}^2$ and a theorem of Bryant on minimal surfaces in $S^4$.  We also classify the associatives in $\Ber$ whose Gauss image lies (\textcolor{red}{tangent spaces lie}?) in an exceptional or singular $\SO(3)$-orbit of the associative Grassmannian.
\end{abstract}

\maketitle

\tableofcontents

%%%%%%%%%%%%%%%%%%%%%%%%%%%%%%%%%%%%%%%%%%%%%%%%%%%%%%%%%%
\section{Introduction}
%%%%%%%%%%%%%%%%%%%%%%%%%%%%%%%%%%%%%%%%%%%%%%%%%%%%%%%%%%

\indent The \textit{Berger space} is the compact homogeneous $7$-manifold
\begin{equation*}
B = \frac{\text{SO}(5)}{\text{SO}(3)},
\end{equation*}
where $\text{SO}(3) \subset \text{SO}(5)$ is a non-standard embedding obtained from the $\text{SO}(3)$-action on the space $\text{Sym}^2_0(\mathbb{R}^3) = \mathbb{R}^5$ of harmonic quadratic forms in three variables.  As an isotropy-irreducible space, $B$ carries a unique $\text{SO}(5)$-invariant metric $g$ up to scale, and this metric is necessarily Einstein.

The space $(B,g)$ was first discovered by Berger \cite{Berger61} in his classification of normal homogeneous metrics of positive curvature.  In fact, homogeneous metrics of positive curvature (not necessarily normal) have been classified by Wallach \cite{Wall72PosCurv} and B\'{e}rard-Bergery \cite{BerBer76PosCurv}, and the complete list is quite short.  Aside from spheres and projective spaces, it consists of sporadic examples in dimensions $6$, $7$, $12$, $13$, and $24$.

%A remarkable property of $g$ was first discovered by Berger \cite{Berger61} in his study of positive sectional curvature.  Indeed, $g$ is one of a small handful of homogeneous metrics with positive sectional curvature. [\textcolor{red}{one more sentence}]  % Alternate Phrasing: The metric $g$ has several remarkable geometric properties.  For example, ......... [positive sectional curvature] [first discovered by Berger] .... \\

Our interest in $(B,g)$, however, comes from the theory of special holonomy.  Indeed, the first known explicit example of an $8$-manifold with $\text{Spin}(7)$ holonomy was $X = \mathbb{R}^+ \times B$ equipped with the cone metric $g_X = dr^2 + r^2 g$ \cite{BryExcept}.  Moreover, $B$ carries a $\text{G}_2$-structure $\varphi \in \Omega^3(B)$ that satisfies
\begin{equation}
d\varphi = 4\ast_\varphi \varphi \label{eq:NP}
\end{equation}

It was later realized \cite{Bar93KSp} that this is an instance of a more general phenomenon.  Namely, given a Riemannian $7$-manifold $(M, g_M)$, the $8$-manifold $X = \mathbb{R}^+ \times M$ equipped with the cone metric $g_X = dr^2 + r^2g_M$ will have holonomy contained in $\text{Spin}(7) \subset \text{SO}(8)$ if and only if $g_M$ is induced from a $\text{G}_2$-structure $\varphi \in \Omega^3(M)$ satisfying (\ref{eq:NP}).  Such $\text{G}_2$-structures are said to be \textit{nearly-parallel} \cite{FrKaMoSe97}, and serve as models for conical singularities of $\text{Spin}(7)$-manifolds. \\

From the point of view of $\G_2$-geometry, the most natural class of submanifolds of a nearly-parallel $\G_2$-manifold $(M, \varphi)$ are the associative submanifolds.  An \textit{associative submanifold} is a $3$-dimensional submanifold $N \to M$ that satisfies the first-order PDE system
\begin{equation*}
\varphi|_N = \text{vol}_N.
\end{equation*}
Although the nearly-parallel equation (\ref{eq:NP}) shows that $\varphi$ is not closed, and hence not a calibration, it is nevertheless the case that associative $3$-folds in $M$ are minimal submanifolds.  Indeed, if $N \subset M$ is an associative $3$-fold, then the $4$-dimensional cone $\mathbb{R}^+ \times N$ inside the $8$-dimensional cone $(\mathbb{R}^+ \times M, dr^2 + r^2g_M)$ will be a Cayley $4$-fold, which is calibrated.  In short, associative $3$-folds in nearly-parallel $\G_2$-manifolds provide models for conically singular Cayley $4$-folds in $\text{Spin}(7)$-manifolds.

Associative 3-folds in nearly parallel $\G_2$-manifolds have been studied by Lotay \cite{LotAssoc}, who considers the case of the $7$-sphere endowed with the standard round metric, and by Kawai \cite{KawAssoc}, who considers the $7$-sphere equipped with the so-called squashed metric. \\

Associative submanifolds in manifolds with \emph{torsion-free} $\G_2$-structure play a fundamental role in efforts to define an invariant of $\G_2$-manifolds by counting \emph{$\G_2$-instantons}, a certain type of gauge theoretical object \cites{DonSeg,DonTho,Joyce18Count,Wal17G2Inst}. In order to deal with technical difficulties arising in this program, it is likely to be necessary to consider gauge theory and associative submanifolds for classes of $\G_2$-structures which satisfy conditions weaker than torsion-free. The \emph{coclosed} $\G_2$-structures, i.e. those satisfying $d *_{\varphi} \varphi = 0,$ are the most natural candidate for such a weaker class, and the nearly parallel $\G_2$-structures form a subclass of these structures.

Gauge theory on nearly parallel $\G_2$-manifolds has been studied by Ball and Oliveira \cite{BaOlAloffWallach} in the case of the Aloff-Wallach manifolds $\SU(3) / \mathrm{S}^1_{k,l},$ and by Waldron \cite{Wal20} in the case of the round 7-sphere $S^7$.

\subsection{Results and Methods} In this work, we study associative submanifolds of the Berger space $B$. We adopt a novel point of view on the Berger space: we think of it as the space of Veronese surfaces in the 4-sphere $S^4$ (see \S\ref{ssect:Veronese}). Our main result is:

\begin{ithm}[\ref{thm:TopType}]\label{ithm:TopType}
	There exist infinitely many topological types of compact (immersed, generically 1-1) $\mathcal{C}$-ruled associative submanifolds of $\Ber.$
\end{ithm}

Topologically, the associative submanifolds in Theorem \ref{ithm:TopType} are circle bundles over genus $g$ surfaces, for every $g \geq 0$.  In fact, they all ruled by a geometrically natural class of geodesics, which we call $\mathcal{C}$-curves.  The space of $\mathcal{C}$-curves in $B$ may be identified with the $8$-manifold $\text{Gr}_2^+(TS^4)$, the Grassmannian of oriented tangent $2$-planes to $S^4$.  Consequently, we can view ruled $3$-folds in $\Ber$ as surfaces in $\text{Gr}^+_2(TS^4).$ Conversely, given a surface $S$ in $\gra$ there is a corresponding ruled 3-fold $\Gamma\!\left(S \right)$ in $\Ber.$  We then ask which surfaces in $\text{Gr}_2^+(TS^4)$ correspond to ruled associatives in $\Ber$.  The answer is given by:

%This perspective highlights a distinguished class of curves in $B$, which we call $\mathcal{C}$-curves.  [Explain $C$-curves.] \\
%We then study associatives in $B$ that are ruled by $C$-curves.  In that direction, in $\S$3 we prove:

\begin{ithm}[\ref{thm:RuledAssoc}]\label{ithm:RuledAssoc}
There is an $\SO(5)$-invariant (non-integrable) almost complex structure $J$ on $\text{Gr}_2^+(TS^4)$ such that:
	\begin{enumerate}
		\item Any ruled associative submanifold of $\Ber$ is locally of the form $\Gamma\!\left( S \right)$ for some $J$-holomorphic curve $S$ in $\mathrm{Gr}_2^+(TS^4)$. 
		\item For each $J$-holomorphic curve $S$ in $\gra$ not locally equivalent to the Gauss lift of a Veronese surface there is a dense subset $S^{\circ} \subset S$ such that $\Gamma\!\left( S^\circ \right)$ is a ruled associative submanifold of $\Ber.$
	\end{enumerate}
\end{ithm}

The proof of Theorem \ref{ithm:TopType} will follow from a construction using the correspondence in Theorem \ref{ithm:RuledAssoc} together with a result of Bryant \cite{Bry82} on superminimal surfaces in $S^4$ and a result of Xu \cite{Xu10} on holomorphic curves in the nearly-K\"ahler $\mathbb{CP}^3$.  One subtlety we encounter is the need to show that the locus $S \setminus S^\circ$ is empty in order to ensure immersion for the resulting associative submanifolds. \\
% Alternate Phrasing: Coupling the correspondence in Theorem B together with a result of Bryant [cite] on super-minimal surfaces in $S^4$ and a result of Xu [cite] on holomorphic curves in the nearly-Kahler $\mathbb{CP}^3$, we deduce are able to deduce Theorem A.  The primary difficulty is ..... \\

In $\S$4, we turn to explicit examples.  For this, we consider the (non-transitive) $\text{SO}(3)$-action on the Grassmannian $\text{Gr}_{\text{ass}}(T_pB)$ of associative $3$-planes at $p \in B$ induced by the isotropy action on $T_pB$.  While the generic associative $3$-plane in $T_pB$ will have trivial stabiliser, larger stabilisers are also possible. We say that an associative submanifold $N$ in $B$ has \emph{special Gauss map} if its tangent planes all have non-trivial stabiliser $\G \leq \SO(3)$. The main result of $\S$4 is the classification of associatives in $B$ with special Gauss map for which the stabiliser $\G$ contains an element of order greater than $2$.  We summarise the classification in the following theorem: %special Gauss map whose tangent planes $T_pN$ all have stabiliser containing an element of order greater than $2$. We summarise the classification in the following theorem:

\begin{ithm}\label{ithm:GaussClass}
	The stabiliser of an associative $3$-plane, if it contains an element of order greater than $2$, is one of:
	\begin{equation*}
	\mathrm{O}(2), \ \ \mathrm{A}_5, \ \ \mathrm{S}_4, \ \ \mathbb{Z}_5, \ \ \mathbb{Z}_4, \ \ \mathbb{Z}_3
	\end{equation*}
	Moreover, suppose that $N$ is an associative 3-fold in $\Ber$ such that $T_pN$ has $\SO(3)$-stabiliser equal to $\G$ for all $p \in N.$
	\begin{itemize}
		\item If $\G = \mathrm{O}(2),$ there are three possibilities corresponding to three distinct $\SO(3)$-orbits on $\mathrm{Gr}_{\mathrm{ass}} \left(T_p \Ber \right).$
		\begin{itemize}[label=$\circ$]
			\item $N$ is a ruled submanifold of $\Ber$ that corresponds via Theorem \ref{ithm:RuledAssoc} to the Gauss lift to $\gra$ of a superminimal surface in $S^4$ or to a fibre of the map $\gra \to \mathbb{CP}^3.$
			\item $N$ is $\SO(5)$-equivalent to (an open subset of) the associative submanifold described in Theorem \ref{thm:145Case}. This associative is homogeneous under an action of $\SO(2) \times \SO(3)$ and diffeomorphic to $\SO(3)/\mathbb{Z}_2.$
			\item $N$ is $\SO(5)$-equivalent to (an open subset of) the associative submanifold described in Theorem \ref{thm:167Case}. This associative is homogeneous under an action of $\mathrm{U}(2)$ and diffeomorphic to $S^3.$
		\end{itemize}
		
	\item If $\G = \mathrm{A}_5$, then $N$ is $\SO(5)$-equivalent to (an open subset of) the associative submanifold described in Theorem \ref{thm:IcoEg}. This associative is homogeneous under an action of $\SO(3)$ and diffeomorphic to the Poincar\'e homology sphere $\SO(3) / \mathrm{A}_5.$
	
	\item If $\G = \mathrm{S}_4$, then $N$ is $\SO(5)$-equivalent to (an open subset of) one of the two associative submanifolds described in Theorem \ref{thm:OctEg}. The first associative is homogeneous under an action of $\SO(3)$ and diffeomorphic to $\SO(3)/\mathrm{S}_4,$ while the second is homogeneous under a different action of $\SO(3)$ and diffeomorphic to $\SO(3) / \mathrm{D}_2$.
	
	\item If $\G = \mathbb{Z}_5,$ then $N$ is a ruled submanifold of $\Ber$ that corresponds via Theorem \ref{ithm:RuledAssoc} to the normal lift to $\gra$  of a superminimal surface in $S^4$.
	
	\item If $\G = \mathbb{Z}_4,$ then $N$ is a ruled submanifold of $\Ber$ that corresponds via Theorem \ref{ithm:RuledAssoc} to the lift of a pseudoholomorphic curve in the nearly-K\"ahler $\mathbb{CP}^3$ to $\gra$ via Proposition \ref{prop:JNK}.
	
	\item There are no associative submanifolds $N$ with tangent spaces having $\SO(3)$-stabiliser everywhere equal to $\mathbb{Z}_3.$
	\end{itemize}
\end{ithm}

In the above result, $\text{A}_5$ and $\text{S}_4$ are the symmetry groups of the icosahedron and octahedron, respectively.

%%%%%%%%%%%%%%

\subsection{Acknowledgments}

We would like to thank Michael Albanese, Robert Bryant, and McKenzie Wang for helpful conversations related to this work. The first author also thanks the Simons Foundation for funding as a
graduate student member of the Simons Collaboration on Special Holonomy in
Geometry, Analysis and Physics for the period during which a large part of this project was completed.

%%%%%%%%%%%%%%%%%%%%%%%%%%%%%%%%%%%%%%%%%%%%%%%%%%%%%%%%%%
\section{Geometry of the Berger space}
%%%%%%%%%%%%%%%%%%%%%%%%%%%%%%%%%%%%%%%%%%%%%%%%%%%%%%%%%%

\subsection{Representation theory of $\SO(3)$}

Let the group $\SO(3)$ act on $\R^3 = \text{span} \lbrace x, y, z \rbrace$ in the usual way. This action extends to an action of $\SO(3)$ on the polynomial ring $\R [ x,y,z]$. Let $\V_n \subset \R [x,y,z]$ be the $\SO(3)$-submodule of homogeneous polynomials of degree $n$, and let $\Hc_n \subset \V_n$ denote the $\SO(3)$-submodule of harmonic polynomials of degree $n$, an irreducible $\SO(3)$-module of dimension $2n+1$. Every finite dimensional irreducible $\SO(3)$-module is isomorphic to $\Hc_n$ for some $n$.

\subsection{The Berger space}\label{ssect:BergSpa}

The irreducible representation $\Hc_2$ of $\SO(3)$ has dimension 5, and thus gives rise to a non-standard embedding $\SO(3) \subset \SO(5).$

\begin{definition}
The \emph{Berger space} is the homogeneous space $\SO(5)/\SO(3)$ given by the quotient of $\SO(5)$ by the copy of $\SO(3)$ described above.
\end{definition}

Topologically, the Berger space is a rational homology sphere with $H^4\!\left(\Ber, \mathbb{Z} \right) = \mathbb{Z}_{10}$ \cite{Berger61} and is diffeomorphic to an $S^3$-bundle over $S^4$ \cite{GoKiSh04}.

The Lie algebra $\mathfrak{so}(5)$ decomposes under the action of $\SO(3)$ as
\begin{align*}
\mathfrak{so}(5) = \mathfrak{so}(3) \oplus \Hc_3,
\end{align*}
and it follows that the Berger space is an isotropy irreducible space. Consequently, $\Ber$ carries a unique $\SO(5)$-invariant metric $g$ up to scale, and this metric is Einstein. In addition, it was shown by Berger \cite{Berger61} that $g$ has positive curvature. It has been shown by Shankar \cite{Sha01Isom} that the isometry group of $\left( \Ber, g \right)$ is exactly $\SO(5).$

\subsection{Veronese surfaces in $S^4$}\label{ssect:Veronese}

When working with a homogeneous space $\G / \mathrm{H}$ it is often advantageous to have an explicit geometric realisation of $\G / \mathrm{H}$ as a set acted on transitively by $\G$ with stabiliser $\mathrm{H}.$ Such a description allows for a geometric understanding of the natural objects, for example distinguished submanifolds, associated to the homogeneous space. In this section we give such an an explicit geometric realisation of the Berger space, in terms of \emph{Veronese surfaces} in the 4-sphere $S^4.$ This realisation will then be used to guide and interpret the calculations and results in the remainder of the paper.

The ring of invariant polynomials on the $\SO(3)$-module $\Hc_2$ has two generators: $g \in \text{Sym}^2 \left( \Hc_2^* \right)$ and $\Upsilon \in \text{Sym}^3 \left( \Hc_2^* \right)$. Let $\left( e_1, \ldots, e_5 \right)$ be the basis for $\Hc_2 \cong \R^5$ given by
\begin{equation}\label{eq:H2Base}
\begin{aligned}
&e_1 =  x^2 - \tfrac{1}{2} y^2 -\tfrac{1}{2} z^2, \\
&e_2 =  \sqrt{3} xy, \:\:\:  e_3 =  \sqrt{3} xz,  \\
&e_4 = \tfrac{\sqrt{3}}{2} y^2- \tfrac{\sqrt{3}}{2} z^2, \:\:\:  e_5 =  \sqrt{3} y z.
\end{aligned}
\end{equation}
In terms of this basis, we have
\begin{align*}
	g \left( v, v \right) & = v_1^2 + v_2^2 + v_3^2 + v_4^2 + v_5^2, \\
	\Upsilon \left( v, v, v \right) & = v_1 \left( v_1^2 + \tfrac{3}{2} v_2^2 + \tfrac{3}{2} v_3^2 - 3 v_4^2 - 3 v_5^2 \right) + \tfrac{3\sqrt{3}}{2} v_4 \left( v_2^2 - v_3^2 \right) + 3 \sqrt{3} v_2 v_3 v_5.
\end{align*}

Consider the subset $\Sigma_0$ in $S^4 \subset \R^5$ defined by the equations
\begin{align*}
g \left( v, v \right) =1, \:\:\: \Upsilon \left( v, v, v \right) = -4.
\end{align*}

\begin{prop}\label{prop:VerInv}
	The subset $\Sigma_0 \subset S^4$ is an embedded copy of $\R \mathbb{P}^2$ invariant under the action of $\SO(3) \subset \SO(5).$
\end{prop}

\begin{proof}
		Identify $\Hc_2$ with the space of traceless symmetric $3 \times 3$ matrices by raising an index. In terms of the basis (\ref{eq:H2Base}), this identification is
		\begin{equation}
		\sum_{i=1}^{5} v_i e_i = \frac{1}{2} \left[ \begin{array}{ccc}
		2 v_1 & \sqrt{3} v_2 & \sqrt{3} v_3 \\
		\sqrt{3} v_2 & -v_1 + \sqrt{3} v_4 & \sqrt{3} v_5 \\
		\sqrt{3} v_3 & \sqrt{3} v_5 & -v_1 - \sqrt{3} v_4
		\end{array} \right].
		\end{equation} 
		The invariant polynomials $g$ and $\Upsilon$ correspond to elementary functions of the eigenvalues of the symmetric matrix. We have
		\begin{equation*}
		\Upsilon = \sigma_3 = 4 \det, \:\:\: g = \tfrac{3}{4} \sigma_2.
		\end{equation*}
		Thus, the subset $\Sigma_0$ is the space of matrices with a repeated positive eigenvalue $1/2.$ The group $\SO(3)$ acts transitively on this space with stabiliser $\mathrm{O}(2).$ Thus, $\Sigma_0 \cong \R \mathbb{P}^2.$
\end{proof}
\begin{remark}
	The subset $\Sigma_0$ may also be characterised as the set of elements of $\Hc_2$ which are the harmonic parts of perfect squares $\left( a_1 x + a_2 y + a_3 z \right)^2$.
\end{remark}
\begin{definition}
	We shall call any embedded surface equivalent to $\Sigma_0$ up to the action of $\SO(5)$ a \emph{Veronese surface} in $S^4.$ We shall call $\Sigma_0$ the \emph{standard Veronese surface} in $S^4.$ 
\end{definition}

Each Veronese surface is an embedded minimal $\R \mathbb{P}^2$ in $S^4$ of constant curvature $1/3.$ The standard Veronese surface $\Sigma_0$ is the image of the immersion $S^2 \to S^4,$
\begin{equation*}
\left( u_1, u_2, u_3 \right) \mapsto \left( u_1^2 - \tfrac{1}{2} u_2^2 -\tfrac{1}{2} u_3^2, \sqrt{3} u_1u_2, \sqrt{3} u_1u_3, \tfrac{\sqrt{3}}{2} u_2^2- \tfrac{\sqrt{3}}{2} u_3^2, \sqrt{3} u_2 u_3 \right),
\end{equation*}
described by Bor\r{u}vka \cite{Boruvka}.

\begin{prop}
The Berger space $\Ber$ is diffeomorphic to the space of Veronese surfaces in $S^4.$
\end{prop}

\begin{proof}
	By construction, the group $\SO(5)$ acts transitively on the space of Veronese surfaces. By proposition \ref{prop:VerInv}, the $\SO(5)$-stabiliser of the standard Veronese surface contains $\SO(3).$ Since the embedding $\SO(3) \subset \SO(5)$ is maximal, it follows that the stabiliser is in fact equal to $\SO(3).$
\end{proof}

\subsection{Structure equations} 
Let $\Sigma$ be a Veronese surface in $S^4.$ Since $\Sigma$ is $\SO(5)$-equivalent to the standard Veronese surface $\Sigma_0,$ there exists a $g$-orthonormal basis $\left( f_1, \ldots, f_5 \right)$ of $\Hc_2$ for which
\begin{equation}\label{eq:SigAdapt}
\Sigma = \left\lbrace w_i f_i \ \vline\  \begin{aligned} & w_1^2 + w_2^2 + w_3^2 + w_4^2 + w_5^2 = 1 \\ & w_1 \!\left( w_1^2 + \tfrac{3}{2} w_2^2 + \tfrac{3}{2} w_3^2 - 3 w_4^2 - 3 w_5^2 \right) + \tfrac{3\sqrt{3}}{2} w_4 \!\left( w_2^2 - w_3^2 \right) + 3 \sqrt{3} w_2 w_3 w_5 = -4 \end{aligned} \right\rbrace.
\end{equation}

\begin{definition}
	Given a Veronese surface $\Sigma,$ we shall say the frame $\left( f_1, \ldots, f_5 \right) \in \SO(5)$ is \emph{$\Sigma$-adapted} if (\ref{eq:SigAdapt}) is satisfied.
\end{definition}

The fibres of the coset projection $\SO(5) \to \Ber$ over a Veronese surface $\Sigma$ are exactly the $\Sigma$-adapted frames and in this way we may think of $\SO(5)$ as a bundle of adapted frames over $\Ber.$

Let $\mathbf{e}_i$ be the i$^{\textup{th}}$ column function on $\SO(5).$ We have the \emph{first structure equation} on $\SO(5):$
\begin{align}\label{eq:SO5Struct}
d \left( \mathbf{e}_1, \ldots, \mathbf{e}_5 \right) = \left( \mathbf{e}_1, \ldots, \mathbf{e}_5 \right) \mu,
\end{align}
where $\mu$ is the $\mathfrak{so}(5)$-valued left-invariant Maurer-Cartan form on $\SO(5).$ Under the splitting $\mathfrak{so}(5) = \mathfrak{so}(3) \oplus \Hc_3,$ we write
\begin{align*}
\mu = \gamma + \omega,
\end{align*}
where $\gamma$ takes values in $\mathfrak{so}(3)$ and $\omega$ takes values in $\Hc_3.$ Explicitly,
\begin{small}
\begin{align}
\gamma &= \left[ \begin {array}{ccccc} 0 & -\sqrt {3}\gamma_{{3}} & \sqrt {3}\gamma_{{2}} & 0 & 0 \\
\sqrt {3}\gamma_{{3}} & 0 & -\gamma_{{1}} & -\gamma_{{3}} & \gamma_{{2}} \\
-\sqrt {3}\gamma_{{2}} & \gamma_{{1}} & 0 & -\gamma_{{2}} & -\gamma_{{3}} \\
0 & \gamma_{{3}} & \gamma_{{2}} & 0 &-2\gamma_{{1}} \\
0 & -\gamma_{{2}} & \gamma_{{3}} & 2\gamma_{{1}} & 0
\end {array} \right], \label{eq:MCSO5} \\
\omega &= \frac{\sqrt{2}}{3} \left[ \begin {array}{ccccc} 0 & -2 \omega_{{2}} & 2 \omega_{{3}} & -\sqrt {10} \omega_{{4}}  & \sqrt {10} \omega_{{5}}  \\
2 \omega_{{2}} & 0 & -2 \sqrt {2} \omega_{{1}} & \sqrt {3} \omega_{{2}} - \sqrt{	5} \omega_{{6}} & -\sqrt {3} \omega_{{3}} + \sqrt {5} \omega_{{7}} \\
-2\omega_{{3}} & 2 \sqrt {2} \omega_{{1}} & 0 & \sqrt {3} \omega_{{3}} + \sqrt{5} \omega_{{7}} & \sqrt {3} \omega_{{2}} + \sqrt {5} \omega_{{6}} \\
 \sqrt {10} \omega_{{4}}  &-\sqrt {3}\omega_{{2}} + \sqrt {5}\omega_{{6}} & -\sqrt{3} \omega_{{3}} - \sqrt {5} \omega_{{7}} & 0 & \sqrt {2}\omega_{{1}} \\
 -\sqrt {10} \omega_{{5}}  & \sqrt {3}\omega_{{3}}-\sqrt {5}\omega_{{7}} & -\sqrt {3} \omega_{{2}} - \sqrt {5} \omega_{{6}} & -\sqrt {2} \omega_{{1}} & 0
\end {array} \right]. \label{eq:MCSO5om}
\end{align}
\end{small}
The $1$-forms $\omega_{{1}}, \ldots, \omega_{{7}}$ are semi-basic for the projection $\SO(5) \to \Ber,$ while the $1$-forms $\gamma_1, \gamma_2, \gamma_3$ are the connection $1$-forms for the homogeneous $\SO(3)$-structure on $\Ber$.

The Maurer-Cartan equation $d \mu= - \mu \wedge \mu$ implies
\begin{small}
\begin{align}
d \left[ \begin{array}{c}
\omega_1 \\
\omega_2 \\
\omega_3 \\
\omega_4 \\
\omega_5 \\
\omega_6 \\
\omega_7
\end{array} \right] =& - \left[ \begin {array}{ccccccc}  0 & \sqrt {6} \gamma_{{2}} & -\sqrt {6} \gamma_{{3}} & 0 & 0 & 0 & 0 \\
-\sqrt {6} \gamma_{{2}} & 0 & \gamma_{{1}} & -\tfrac{\sqrt{10}}{2} \gamma_{{3}} &-\tfrac{\sqrt{10}}{2} \gamma_{{2}} & 0 & 0 \\
\sqrt {6} \gamma_{{3}} & - \gamma_{{1}} & 0 & \tfrac{\sqrt{10}}{2} \gamma_{{2}} &-\tfrac{\sqrt{10}}{2} \gamma_{{3}} & 0 & 0 \\
0 & \tfrac{\sqrt{10}}{2} \gamma_{{3}} & -\tfrac{\sqrt{10}}{2} \gamma_{{2}} & 0 & 2 \gamma_{{1}} & - \tfrac{\sqrt{6}}{2} \gamma_{{3}} & -\tfrac{\sqrt{6}}{2} \gamma_{{2}} \\
 0 & \tfrac{\sqrt{10}}{2} \gamma_{{2}} & \tfrac{\sqrt{10}}{2} \gamma_{{3}} & -2 \gamma_{{1}} & 0 & \tfrac{\sqrt{6}}{2} \gamma_{{2}} & -\tfrac{\sqrt{6}}{2} \gamma_{{3}} \\
 0 & 0 & 0 & \tfrac{\sqrt{6}}{2} \gamma_{{3}} & -\tfrac{\sqrt{6}}{2} \gamma_{{2}} & 0 & 3\gamma_{{1}} \\
 0 & 0 & 0 & \tfrac{\sqrt{6}}{2} \gamma_{{2}} & \tfrac{\sqrt{6}}{2} \gamma_{{3}} & -3 \gamma_{{1}} & 0 \end {array} \right] \wedge \left[ \begin{array}{c}
 \omega_1 \\
 \omega_2 \\
 \omega_3 \\
 \omega_4 \\
 \omega_5 \\
 \omega_6 \\
 \omega_7
 \end{array} \right] \nonumber \\
 & + \frac{2}{3} \left[ \begin{array}{c}
  - \omega_2 \wedge \omega_3 - \omega_4 \wedge \omega_5 + \omega_6 \wedge \omega_7 \\
 \omega_1 \wedge \omega_3 - \omega_4 \wedge \omega_6 - \omega_5 \wedge \omega_7 \\
 -\omega_1 \wedge \omega_2 + \omega_5 \wedge \omega_6 - \omega_4 \wedge \omega_7 \\
 \omega_1 \wedge \omega_5 + \omega_2 \wedge \omega_6 + \omega_3 \wedge \omega_7 \\
 -\omega_1 \wedge \omega_4 + \omega_3 \wedge \omega_6 + \omega_2 \wedge \omega_7 \\
 -\omega_1 \wedge \omega_7 - \omega_2 \wedge \omega_4 + \omega_3 \wedge \omega_5 \\
 \omega_1 \wedge \omega_6 - \omega_2 \wedge \omega_5 - \omega_3 \wedge \omega_4
 \end{array}
 \right] \label{eq:BerStruct1}
\end{align}
\end{small}
and
\begin{small}
\begin{align}\label{eq:BerStruct2}
\left[ \begin{array}{c}
d \gamma_1 + \gamma_2 \wedge \gamma_3 \\
d \gamma_2 + \gamma_3 \wedge \gamma_1 \\
d \gamma_3 + \gamma_1 \wedge \gamma_2
\end{array} \right] = \frac{2}{9} \left[ \begin{array}{c}
2 \omega_2 \wedge \omega_3 + 4 \omega_4 \wedge \omega_5 + 6 \omega_6 \wedge \omega_7 \\
2 \sqrt{6} \omega_1 \wedge \omega_2 - \sqrt{10} \left( \omega_2 \wedge \omega_5 - \omega_3 \wedge \omega_4 \right) - \sqrt{6} \left( \omega_4 \wedge \omega_7 - \omega_5 \wedge \omega_6 \right) \\
-2 \sqrt{6} \omega_1 \wedge \omega_3 - \sqrt{10} \left( \omega_2 \wedge \omega_4 + \omega_3 \wedge \omega_5 \right) - \sqrt{6} \left( \omega_4 \wedge \omega_6 + \omega_5 \wedge \omega_7 \right)
\end{array} \right],
\end{align}
\end{small}
which we shall refer to as the structure equations of $\Ber.$

\subsubsection{The $\G_2$-structure on $\Ber$}
Consider the 3-form $\tilde{\varphi}$ on $\SO(5)$ defined by
\begin{align}\label{eq:G2StructBer}
\tilde{\varphi} = \omega_{123} + \omega_{145} - \omega_{167} + \omega_{256} + \omega_{257} + \omega_{347} - \omega_{356}
\end{align}
By the structure equations (\ref{eq:BerStruct1}), $\tilde{\varphi}$ is invariant under the action of $\SO(3) \subset \SO(5).$ Consequently, $\tilde{\varphi}$ is the pullback to $\SO(5)$ of a 3-form $\varphi$ on $\Ber.$ The 3-form $\varphi$ defines a $\G_2$-structure on $\Ber.$ The induced metric $g_{\varphi}$ and 4-form $*_{\varphi} \varphi$ on $\Ber$ satisfy
\begin{align*}
\pi^* \left( g_{\varphi} \right) &= \omega_1^2 + \omega_2^2 + \omega_3^2 + \omega_4^2 + \omega_5^2 + \omega_6^2 + \omega_7^2, \\
\pi^* \left( *_{\varphi} \varphi \right) &= \omega_{4567} + \omega_{2367} - \omega_{2345} + \omega_{1357} + \omega_{1346} + \omega_{1256} - \omega_{1247},
\end{align*}
where $\pi$ denotes the projection $\SO(5) \to \Ber.$ The metric $g_{\varphi}$ is $\SO(5)$-invariant and, since $B$ is isotropy-irreducible, it agrees with the metric $g$ of \S\ref{ssect:BergSpa} up to scale.

The structure equations (\ref{eq:BerStruct1}) imply that $\varphi$ satisfies
\begin{align*}
d \varphi = 4 *_{\varphi} \varphi,
\end{align*}
so $\varphi$ defines a \emph{nearly parallel $\G_2$-structure} on $\Ber.$ It follows that the metric $g_{\varphi}$ is Einstein with scalar curvature 42, and that the metric cone over $\Ber$ has holonomy contained in $\mathrm{Spin}(7).$ In fact, the cone over $\Ber$ was the first explicit example of a metric with $\mathrm{Spin}(7)$-holonomy \cite{BryExcept}. For more details on nearly parallel $\G_2$-structures see \cite{FrKaMoSe97}.

\subsection{Associative submanifolds} In this work, we will study a special class of $3$-dimensional submanifolds of the Berger space known as associative submanifolds.

\begin{definition} Let $M^7$ be an oriented $7$-manifold equipped with a $\text{G}_2$-structure $\varphi \in \Omega^3(M)$.  An oriented $3$-dimensional submanifold $N \subset M$ is called an \emph{associative $3$-fold} if:
\begin{equation*}
\varphi|_N = \text{vol}_N.
\end{equation*}
\end{definition}

Two special cases are worth highlighting.  First, if $d\varphi = 0$, then $\varphi$ is a calibration \cite{HarLawCali}, and hence associative $3$-folds in $M$ are area-minimizing.  Second, if $\varphi$ is nearly-parallel, meaning that $d\varphi = 4\ast_\varphi \varphi$, then associative $3$-folds in $M$ are the links of Cayley cones in the metric cone over $M$, and hence are also minimal submanifolds of $M$. In fact, the two special cases just described are exactly the classes of $\G_2$-structures for which every associative 3-fold is minimal \cite{BaMaExcept}.

Associative $3$-folds in $7$-manifolds with nearly-parallel $\text{G}_2$-structures have been studied by Lotay \cite{LotAssoc}, who considers the round $7$-sphere, and Kawai \cite{KawAssoc}, who considers the squashed $7$-sphere.

\subsection{Subgroups and Quotients of $\SO(5)$} Several different homogeneous spaces of $\SO(5)$ will play a role in this work.  As a guide to these, and to fix conventions, we indicate the connected subgroups $\mathrm{H}$ of $\SO(5)$ up to $\SO(5)$-conjugacy in Figure \ref{fig:SO5Grp}. The corresponding diagram of homogeneous spaces $\SO(5)/\mathrm{H}$ is given in Figure \ref{fig:SO5HomSp}.
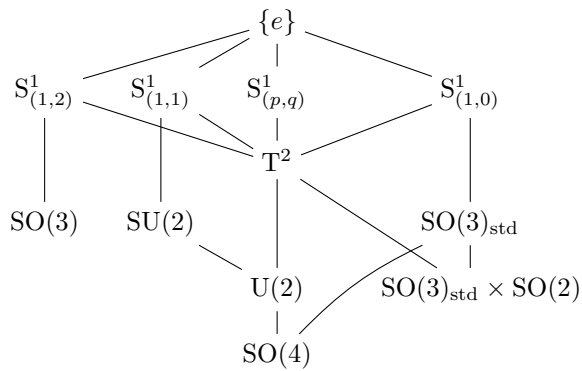
\begin{figure}\caption{Connected subgroups of $\SO(5)$ up to conjugacy}\label{fig:SO5Grp}
	$$\begin{tikzcd}[column sep=small, row sep=small,cramped]
	&                                                       & \{e\} \arrow[lld, no head] \arrow[ld, no head] \arrow[d, no head] \arrow[rd, no head] &                                                                     \\
	{\Sg^1_{(1,2)}} \arrow[rrd, no head] \arrow[dd, no head] & {\Sg^1_{(1,1)}} \arrow[rd, no head] \arrow[dd, no head] & {\Sg^1_{(p,q)}} \arrow[d, no head]                                                      & {\Sg^1_{(1,0)}} \arrow[ld, no head] \arrow[dd, no head]               \\
	&                                                       & \mathrm{T}^2 \arrow[dd, no head] \arrow[rdd, no head]                                 &                                                                     \\
	\mathrm{SO}(3)                                         & \mathrm{SU}(2) \arrow[rd, no head]                    &                                                                                       & \mathrm{SO}(3)_{\text{std}} \arrow[d, no head] \arrow[ldd, no head, bend right = 10] \\
	&                                                       & \mathrm{U}(2) \arrow[d, no head]                                                      & \ \ \, \mathrm{SO}(3)_{\text{std}} \times \mathrm{SO}(2)                   \\
	&                                                       & \mathrm{SO}(4)                                                                        &                                                                    
	\end{tikzcd}$$
\end{figure}

\begin{figure}\caption{Homogeneous spaces of $\SO(5)$}\label{fig:SO5HomSp}
	$$	\begin{tikzcd}[column sep=small, row sep=small,cramped]
	&     & \mathrm{SO}(5) \arrow[lld, no head] \arrow[ld, no head] \arrow[d, no head] \arrow[rd, no head] &  \\
	{\mathrm{SO}(5)/\Sg^1_{(1,2)}} \arrow[rrd, no head] \arrow[dd, no head] & {\mathrm{SO}(5)/\Sg^1_{(1,1)}} \arrow[rd, no head] \arrow[dd, no head] & {\mathrm{SO}(5)/\Sg^1_{(p,q)}} \arrow[d, no head]    & {\mathrm{SO}(5)/\Sg^1_{(1,0)}} \arrow[ld, no head] \arrow[dd, no head] \\
	&   & \text{Gr}_2^+(TS^4) \arrow[dd, no head] \arrow[rdd, no head]    &  \\
	B = \mathrm{SO}(5)/\mathrm{SO}(3)  & S^7 \arrow[rd, no head]   & & T_1(S^4) \arrow[d, no head] \arrow[ldd, no head]                     \\
	&   & \mathbb{CP}^3 \arrow[d, no head]      & \text{Gr}_2^+(\mathbb{R}^5)         \\
	&   & S^4       &                                                                     
	\end{tikzcd}$$
\end{figure}
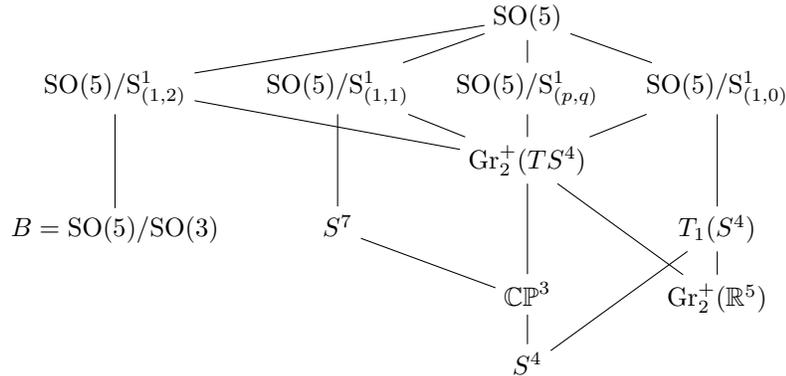
%$$ \begin{tikzcd}
% &   & \{e\} \arrow[lld, no head] \arrow[ld, no head] \arrow[d, no head] \arrow[rrd, no head] &  &  \\
%{S^1_{(1,2)}} \arrow[rrd, no head] \arrow[dd, no head] & {S^1_{(1,1)}} \arrow[rd, no head] \arrow[dd, no head] & {S^1_{(p,q)}} \arrow[d, no head]   &  & {S^1_{(1,0)}} \arrow[dd, no head]      \\
% &   & \mathrm{T}^2 \arrow[rru, no head] \arrow[dd, no head] \arrow[rrdd, no head]    &  &  \\
%\SO(3)  & \mathrm{SU}(2) \arrow[rd, no head]   &    &  & \mathrm{SO}(3)_{\text{std}} \arrow[d, no head] \arrow[lldd, no head] \\
%   &   & \mathrm{U}(2) \arrow[d, no head]   &  & \mathrm{SO}(3)_{\text{std}} \times \mathrm{SO}(2)                    \\
%  &    & \mathrm{SO}(4)                                                                         &  &                                                        
%\end{tikzcd}$$

In Figure \ref{fig:SO5Grp}, for $(p,q) \in \mathbb{Z}^2$, $(p,q) \neq (0,0)$, we let  $\Sg^1_{p,q} \leq \SO(5)$ denote the circle subgroup given by
\begin{equation*}
\theta \cdot (e_1, e_2 + ie_3, e_4 + ie_5) = (e_1, e^{ip\theta}(e_2 + ie_3), e^{iq\theta}(e_4 + ie_5)).
\end{equation*}
Each of these circles are subgroups of the maximal torus $\mathrm{T}^2$ of $\SO(5)$ given by
\begin{equation} \label{eq:maxtorus}
(\theta, \phi) \cdot (e_1, e_2 + ie_3, e_4 + ie_5) = (e_1, e^{i\theta}(e_2 + ie_3), e^{i\phi}(e_4 + ie_5)).
\end{equation}
We let $\mathrm{O}(2)_{p,q}$ denote the group generated by $\Sg^1_{p,q}$ and the element $\mathrm{diag} \left(1, 1, -1, 1, -1\right).$

The group $\SO(4) \subset \SO(5)$ is defined to be the identity component of the group fixing the vector $e_1,$ and the subgroup $\mathrm{U}(2) \subset \SO(4)$ is the subgroup fixing the 2-form $e^{23} + e^{45}.$ The subgroup $\mathrm{SU}(2) \subset \mathrm{U}(2)$ is the subgroup fixing $\left( e_2 + i e_3 \right) \wedge \left(e_4 + i e_5 \right).$

The subgroup $\SO(2) \times \SO(3)_{\mathrm{std}} \subset \SO(5)$ is the identity component of the group preserving the 3-plane $\mathrm{span} \left( e_3, e_4, e_5 \right),$ and the subgroup $\SO(3)_{\mathrm{std}} \subset \SO(2) \times \SO(3)_{\mathrm{std}}$ is the subgroup acting trivially on $\mathrm{span} \left(e_1, e_2 \right).$

In Figure \ref{fig:SO5HomSp}, $T_1(S^4)$ denotes the unit tangent bundle of $S^4$, while $\text{Gr}_2^+(\mathbb{R}^5)$ denotes the Grassmannian of oriented $2$-planes in $\mathbb{R}^5$, and $\text{Gr}_2^+(TS^4)$ denotes the Grassmann bundle of oriented tangent $2$-planes to $S^4$.

%%%%%%%  If you don't like the circle subgroups: The commented text below gives the corresponding diagrams without the circle subgroups.

%%%%%%%

% The following diagram lists the connected subgroups $H$ of $\SO(5)$ having $\dim(H) \geq 2$, up to $\SO(5)$-conjugacy.

%$$\begin{tikzcd}
%    &  &  & \{e\} \arrow[llldd, no head] \arrow[d, no head] \arrow[rrdd, no head] \arrow[ldd, no head]  &  &    \\
%    &  &  & T^2  \arrow[dd, no head] \arrow[rrdd, no head]   &  &    \\
%\SO(3)_{\text{irr}} &  & \mathrm{SU}(2) \arrow[rd, no head] & & & \SO(3) \arrow[d, no head] \arrow[lldd, no head] \\
%     &  &  & \mathrm{U}(2) \arrow[d, no head]  &  & \SO(3) \times \SO(2)  \\
%     &  &  & \SO(4)  &  &                                                      
%\end{tikzcd}$$
%\noindent The corresponding diagram of homogeneous spaces $\SO(5)/H$ is as follows:

%$$\begin{tikzcd}
%  & & & \text{SO}(5) \arrow[llldd, "\SO(3)"'] \arrow[d, "T^2"] \arrow[rrdd, "\SO(3)"] \arrow[ldd, "S^3"'] & & \\
%  & & & \text{Gr}_2^+(T\mathbb{S}^4) \arrow[dd, "S^2"] \arrow[rrdd, "S^2"]  &  &   \\
%\Ber_{\text{irr}} &  & \mathbb{S}^7 \arrow[rd, "S^1"']  & &  & V_2(\mathbb{R}^5) \arrow[d, "S^1"] \arrow[lldd, "S^3"] \\
%  &  & & \mathbb{CP}^3 \arrow[d, "S^2"]     &  & \text{Gr}_2^+(\mathbb{R}^5)  \\
%  &  & & S^4 &  &                                                       
%\end{tikzcd}$$
% In this diagram, labels on arrows indicate the fibers of the projection map.

\subsection{Cohomogeneity-one action of $\SO(4)$}

The action of the subgroup $\SO(4) \subset \SO(5)$ on the Berger space $\Ber$ is cohomogeneity-one, meaning that its principal orbits have codimension 1. This action was first described in the context of manifolds of positive curvature by Verdiani-Podest\`a \cite{PoVe99}, and appears in the classification of simply-connected positively curved cohomogeneity-one manifolds due to Grove-Wilking-Ziller \cite{GWZ08}. Since the nearly parallel $\G_2$-structure $\varphi$ on $\Ber$ defined above is $\SO(5)$-invariant, it is \emph{a fortiori} invariant under this cohomogeneity-one action.

The $\SO(4)$-orbit through the identity coset $\mathrm{Id}_5 \, \SO(3) = \Sigma_0$ is a singular orbit. The curve $u : \R \to \SO(5),$ 
\begin{align*}
u : s \mapsto \begin{small} \left[ \begin{array}{ccccc}
\cos \frac{2 \sqrt{5}}{3} s & 0 & 0 & -\sin \frac{2 \sqrt{5}}{3} s & 0 \\
0 & 0 & 0 & 0 & 0 \\
0 & 0 & 0 & 0 & 0 \\
\sin \frac{2 \sqrt{5}}{3} s & 0 & 0 & \cos \frac{2 \sqrt{5}}{3} s & 0 \\
0 & 0 & 0 & 0 & 0
\end{array} \right], \end{small}
\end{align*}
projects under the map $ \pi: \SO(5) \to \Ber$ to a geodesic orthogonal to all $\SO(4)$-orbits, and the image of the map \begin{align*}
% \left[ 0, \pi / \sqrt{5} \right] \times \SO(4) \ni (s, A) \mapsto A \, u(s) \in \SO(5)
\left[ 0, \pi / \sqrt{5} \right] \times \SO(4) & \to \SO(5) \\
 (s, A) & \mapsto A \, u(s)
\end{align*}
surjects onto $\Ber.$ The $\SO(4)$-stabiliser of the point $\pi \left( u\left( s \right) \right) \in \Ber$ is given by
\begin{equation*}
\left( u(s)^{-1} \SO(4) u(s) \right) \cap \SO(3) = \mathrm{Stab}_{\SO(3)} \left(u(s)^{-1} e_1 \right) \cong \begin{cases}
\mathrm{O}(2)_{1,2} & \text{if} \: s = 0, \pi / \sqrt{5}, \\
\mathbb{Z}_2 \times \mathbb{Z}_2 & \text{otherwise}.
\end{cases}
\end{equation*}
Thus, the group picture for the action of $\SO(4)$ on $\Ber$ is
\begin{align*}
\mathbb{Z}_2 \times \mathbb{Z}_2 \subset \left\lbrace \mathrm{O}(2)_{1,2} , \mathrm{O}(2)_{1,2} \right\rbrace \subset \SO(4).
\end{align*}

 Writing the Maurer-Cartan form of $\SO(4) \subset \SO(5)$ in a manner adapted to the splitting $\mathfrak{so}(4) \cong \mathfrak{su}(2) \oplus \mathfrak{su}(2)$ as
\begin{small}
	\begin{align*}
	\frac{1}{2} \left[ \begin {array}{ccccc} 0&0&0&0&0\\ 0&0&-\mu_{
		{1}}+\nu_{{1}}&-\mu_{{2}}+\nu_{{2}}&-\mu_{{3}}+\nu_{{3}}
	\\ 0&\mu_{{1}}-\nu_{{1}}&0&-\mu_{{3}}-\nu_{{3}}&\mu_
	{{2}}+\nu_{{2}}\\ 0&\mu_{{2}}-\nu_{{2}}&\mu_{{3}}+
	\nu_{{3}}&0&-\mu_{{1}}-\nu_{{1}}\\ 0&\mu_{{3}}-\nu_{
		{3}}&-\mu_{{2}}-\nu_{{2}}&\mu_{{1}}+\nu_{{1}}&0\end {array} \right],
	\end{align*}
\end{small}
the pullback of the Maurer-Cartan form of $\SO(5)$ to $\left[ 0, \pi /3 \right] \times \SO(4)$ is given by
\begin{small}
	\begin{align*}
	\frac{1}{2} \left[ \begin {array}{ccccc} 0 & \sin \frac{2 \sqrt{5}}{3} s \left( \mu_2 - \nu_2 \right) & \sin  \frac{2 \sqrt{5}}{3} s \left( \mu_{{3}} + \nu_{{3}} \right) &- \frac{2 \sqrt{5}}{3} \, d s & -\sin \frac{2 \sqrt{5}}{3} s \left( \mu_{{1}} + \nu_{{1}} \right) \\
	\sin  \frac{2 \sqrt{5}}{3} s  \left( -\mu_{{2}} + \nu_2 \right) & 0 & -\mu_{{1}}+\nu_{{1}} & \cos \frac{2 \sqrt{5}}{3} s  \left( -\mu_{{2}} + \nu_{{2}} \right) & -\mu_{{3}} + \nu_{{3}} \\
	-\sin  \frac{2 \sqrt{5}}{3} s \left( \mu_{{3}} + \nu_{{3}} \right) & \mu_{{1}}-\nu_{{1}} & 0 & - \cos \frac{2 \sqrt{5}}{3} s \left( \mu_{{3}} + \nu_{{3}} \right) & \mu_{{2}} + \nu_{{2}} \\
	\frac{2 \sqrt{5}}{3} d  s & \cos \frac{2 \sqrt{5}}{3} s \left( \mu_{{2}} -\nu_{{2}} \right) & \cos \frac{2 \sqrt{5}}{3} s \left( \mu_{{3}} + \nu_{{3}} \right) & 0 & \cos  \frac{2 \sqrt{5}}{3} s  \left( -\mu_{{1}}- \nu_{{1}} \right) \\ 
	\sin \frac{2 \sqrt{5}}{3} s \left( \mu_{{1}} + \nu_{{1}} \right)  & \mu_{{3}}-\nu_{{3}} & -\mu_{{2}}-\nu_{{2}} & \cos  \frac{2 \sqrt{5}}{3} s \left( \mu_{{1}} + \nu_{{1}} \right) & 0 \end {array} \right].
	\end{align*}
\end{small}

Reparametrising $s = (3 \sqrt{5} / 10 ) t$, on $\left[ 0, \pi /3 \right] \times \SO(4),$ we have that
\begin{equation}\label{eq:PrinOrbitsoms}
\begin{aligned}
\omega_{{1}} &= \tfrac{3}{20} \left( \left( 2 - \cos t \right) \mu_1 - \left( 2 + \cos t \right) \nu_1 \right), \\
\omega_{{2}} &= \tfrac{3 \sqrt{2}}{40} \left( \left( \sqrt{3} - \sqrt{7} \cos \left( t - \phi \right) \right) \mu_2 + \left( \sqrt{3} + \sqrt{7} \cos \left( t - \phi \right) \right)  \nu_2 \right), \\
\omega_{{3}} &= \tfrac{3 \sqrt{2}}{40} \left( \left( \sqrt{3} - \sqrt{7} \cos \left( t + \phi \right) \right) \mu_3 - \left( \sqrt{3} + \sqrt{7} \cos \left( t + \phi \right) \right)  \nu_3 \right), \\
\omega_4 &= \tfrac{3 \sqrt{5}}{10} d t, \\
\omega_5 &= - \tfrac{3 \sqrt{5}}{20} \sin t \left( \mu_1 - \nu_1 \right) \\
\omega_6 &= \tfrac{3 \sqrt{10}}{40} \left( \left( 1 + \cos t \right) \mu_2 + \left( 1 - \cos t \right) \nu_2 \right) \\
\omega_7 &= \tfrac{3 \sqrt{10}}{40} \left( \left( -1 - \cos t \right) \mu_3 + \left( 1 - \cos t \right) \nu_3 \right),
\end{aligned}
\end{equation}
where $\phi = \arctan \left( {2} / {\sqrt{3}} \right).$ Pulling back equation (\ref{eq:G2StructBer}) and using the above formulas gives an explicit expression for the nearly parallel $\G_2$-structure $\varphi$ on $\Ber$ as a curve in the space of invariant 3-forms on the principal orbits $\SO(4) / \mathbb{Z}^2_2.$

\subsubsection{$\SU(3)$-structure on principal orbits}

The group $\G_2$ acts transitively on the 6-sphere $S^6$ with stabiliser $\SU(3),$ and thus any hypersurface $X^6$ in a manifold $M$ with $\G_2$-structure has an induced $\SU(3)$ structure $ \left(\Omega, \Re \Upsilon \right) \in \Omega^2\left(X \right) \oplus \Omega^3 \left( X \right),$ given explicitly by
\begin{equation*}
\Omega = v \lrcorner \varphi, \:\:\: \Re \Upsilon = \varphi |_X,
\end{equation*}
where $v$ is a unit normal vector field to $X.$

The torsion of the induced $\SU(3)$-structure on the $X$ is determined by the torsion of the ambient $\G_2$-structure and the second fundamental form of the inclusion $X \subset M$. When the ambient $\G_2$-structure is nearly parallel, the induced $\SU(3)$-structure on a hypersurface is of a type called \emph{nearly half-flat} \cite{FIMU08}, meaning that $d \Re \Upsilon = - 2 \Omega \wedge \Omega.$ Thus, the principal orbits $\SO(4) / \mathbb{Z}_2^2$ of the $\SO(4)$-action on $\Ber$ all carry nearly half-flat $\SU(3)$-structures. Explicitly, the $\SU(3)$-structure on a principal orbit $\left\lbrace s = \mathrm{const.} \right\rbrace$ is given by the forms
\begin{equation}\label{eq:PrinOrbitStruct}
\begin{aligned}
\Omega &= - \omega_{15} - \omega_{26} - \omega_{37}, \\
\Re \Upsilon &= \omega_{123} - \omega_{167} + \omega_{257} - \omega_{356},
\end{aligned}
\end{equation}
restricted to this orbit.

\begin{prop}\label{prop:PrinOrbSLag}
	Any special Lagrangian submanifold (of phase 0) of a principal orbit $\SO(4) / \mathbb{Z}_2^2 \subset \Ber$ is an associative submanifold of $\Ber.$
\end{prop}

\begin{proof}
	This is true for any special Lagrangian submanifold in a hypersurface of a manifold with $\G_2$-structure, as we now show. An oriented 3-dimensional submanifold $N$ of a 6-manifold $X$ with $\SU(3)$-structure $\left(\Omega, \Re \Upsilon \right)$ is special Lagrangian (of phase 0) if and only if $\Re \Upsilon |_N = \mathrm{vol}|_N.$ If the $\SU(3)$-structure on $X$ is induced from an inclusion $X \subset M$ of $X$ as a hypersurface in a 7-manifold $M$ with $\G_2$-structure $\varphi,$ then $\Re \Upsilon = \varphi |_X.$ Thus, $\varphi |_N = \mathrm{vol}_N,$ so $N$ is an associative submanifold of $M.$
\end{proof}

Proposition \ref{prop:PrinOrbSLag} is not of direct use in constructing associative submanifolds of $\Ber,$ as typically there will be obstructions to the existence of special Lagrangian submanifolds in the principal orbits arising from the torsion of the induced $\SU(3)$-structure \cite{BaMaSLag}. In particular, it is possible to compute from equations (\ref{eq:PrinOrbitsoms}) and (\ref{eq:PrinOrbitStruct}) that the induced $\SU(3)$-structure on a principal orbit is never nearly K\"ahler.

%%%%%%%%%%%%%%%%%%%%%%%%%%%%%%%%%%%%%%%%%%%%%%%%%%%%%%%%%%
\section{Ruled associative submanifolds}\label{sect:Ruled}
%%%%%%%%%%%%%%%%%%%%%%%%%%%%%%%%%%%%%%%%%%%%%%%%%%%%%%%%%%

A natural condition on a submanifold is that it be ruled by some special class of curves in the ambient space. In the context of calibrated geometry, special Lagrangian submanifolds of $\mathbb{C}^3$ ruled by lines have been studied by Bryant \cite{Bry06} and Joyce \cite{Joyce02Ruled}. Fox considers coassociative cones in $\R^7$ and Cayley cones in $\R^8$ ruled by 2-planes \cites{Fox2008Cayley,Fox2007ConesPlanes}. Lotay has also studied coassociative submanifolds of $\R^7$ and Cayley submanifolds of $\R^8$ ruled by 2-planes, as well as special Lagrangians in $S^6$ and associative submanifolds of $S^7$ ruled by circles \cites{LotAssoc, LotLag11}. In this section, we shall apply similar techniques to the Berger space $\Ber.$ 

Unlike in the cases described above, where the ambient manifold is a Euclidean space or a sphere, it is not obvious what the appropriate class of ruling curves should be. Our first step is then to describe the ruling curves we shall consider. Consider the distribution $\mathcal{D}$ on $\SO(5)$ given by
\begin{align*}
\mathcal{D} = \text{ker} \left( \omega_2, \omega_3, \ldots, \omega_7, \gamma_2, \gamma_3 \right).
\end{align*}
 By the structure equations (\ref{eq:BerStruct1}), $\mathcal{D}$ is a Frobenius system.  In fact, using the structure equation (\ref{eq:SO5Struct}), we can explicitly describe the integral surface of $\mathcal{D}$ passing through the frame $\left( f_1, \ldots, f_5 \right) \in \SO(5)$: it is the image of the map $T^2 \to \SO(5)$ given by
	 \begin{equation}\label{eq:D-Surf}
	 \begin{aligned}
	 \left( \theta, \phi \right) \mapsto \left( f_1, \cos(\theta + 2 \phi ) f_2 + \sin(\theta + 2 \phi ) f_3, -\sin(\theta + 2 \phi ) f_2 + \cos(\theta + 2 \phi ) f_3, \right. \\
	 \left. \cos(2 \theta - \phi ) f_4 + \sin(2 \theta - \phi ) f_5, -\sin(2 \theta - \phi ) f_4 + \cos(2 \theta - \phi ) f_5 \right).
	 \end{aligned}
	 \end{equation}
% The integral surfaces of $\mathcal{D}$ project to homogeneous curves in the Berger space $\Ber$. We shall call such curves $\C$-curves in $\Ber.$
In particular, we see that the (connected, maximal) integral surfaces of $\mathcal{D}$ in $\SO(5)$ (hereafter called \emph{$\mathcal{D}$-surfaces}) are precisely the cosets of the maximal torus $\mathrm{T}^2$ in $\SO(5)$ chosen in (\ref{eq:maxtorus}).  Said another way, the $\mathcal{D}$-surfaces are precisely the fibers of the map
\begin{equation}\label{eq:lambda-map}
\begin{aligned} 
\lambda \colon \SO(5) & \to \text{Gr}_2^+(TS^4) \cong \SO(5)/\mathrm{T}^2 \\
\lambda(f_1, \ldots, f_5) & = (f_1, \text{span}(f_2, f_3))
\end{aligned}
\end{equation}

 The $\mathcal{D}$-surfaces in $\SO(5)$ project to curves in the Berger space $\Ber$ that we shall call $\C$-curves. We will consider $\mathcal{C}$-curves as oriented.  We will see shortly (Proposition \ref{prop:Ccurves}) that the $\C$-curve given by the projection of the $\mathcal{D}$-surface (\ref{eq:D-Surf}), for example, is the family of Veronese surfaces $\Sigma$ that satisfy $f_1 \in \Sigma$ and $T_{f_1}\Sigma = \text{span}(f_2, f_3)$.  We also note that every $\C$-curve is an orbit of the subgroup $S^1_{(-2,1)} \leq \SO(5)$, or of one of its conjugates.

 \begin{definition}\label{defn:ruled}
 	An associative submanifold of $\Ber$ is said to be \emph{ruled} if it is foliated by $\C$-curves.
 \end{definition} 
 
 \begin{remark}
 	The $\SO(3)$-structure on $\Ber$ naturally identifies each tangent space of $\Ber$ with the $\SO(3)$-module $\mathcal{H}_3.$ The action of $\SO(3)$ on $\Hc_3$ preserves the 3-dimensional cone of harmonic cubics that are the harmonic parts of a perfect cube $(a_1 x + a_2y + a_3)^3$.  Explicitly, these are the harmonic cubics of the form
 	\begin{equation}\label{eq:Celts}
 	\left( a_1 x + a_2 y + a_3 z \right)^3 - \tfrac{3}{5} \left( a_1^2 + a_2^2 + a_3^2 \right) \left( a_1 x + a_2 y + a_3 z \right) \left( x^2 + y^2 +z ^2 \right)
 	\end{equation}
 	for $\left( a_1, a_2, a_3 \right) \in \R^3.$ Thus, the tangent bundle $T \Ber$ contains an $\SO(5)$-invariant subset $\mathcal{C}$ consisting of the tangent vectors identified with the elements of the form (\ref{eq:Celts}). The $\mathcal{C}$-curves in $\Ber$ are simply the geodesics that are everywhere tangent to $\mathcal{C}.$ 
 \end{remark}
 
 \begin{prop}\label{prop:Ccurves}
 	The space of $\C$-curves in $\Ber$ is diffeomorphic to the flag manifold $\SO(5)/\mathrm{T}^2 \cong \text{Gr}_{2}^+\!\left( T S^4 \right)$.  An explicit $\SO(5)$-equivariant diffeomorphism is given by the map
 	\begin{align*}
 		\Gamma \colon \text{Gr}_2^+(TS^4) & \to \{\mathcal{C}\text{-curves in } \Ber \} \\
		\Gamma(p , E) & = \left\lbrace \Sigma \in \Ber \mid p \in \Sigma, \: T_p \Sigma = E \right\rbrace.
 	\end{align*}
 \end{prop}

\begin{proof}
%	 The integral surface of $\mathcal{D}$ passing through the frame $\left( f_1, \ldots, f_5 \right) \in \SO(5)$ may be determined using the structure equation (\ref{eq:SO5Struct}): it is the image of the map $T^2 \to \SO(5)$ given by
%	 \begin{align*}
%	 \left( \theta, \phi \right) \mapsto \left( f_1, \cos(\theta + 2 \phi ) f_2 + \sin(\theta + 2 \phi ) f_3, -\sin(\theta + 2 \phi ) f_2 + \cos(\theta + 2 \phi ) f_3, \right. \\
%	 \left. \cos(2 \theta - \phi ) f_4 + \sin(2 \theta - \phi ) f_5, -\sin(2 \theta - \phi ) f_4 + \cos(2 \theta - \phi ) f_5 \right).
%	 \end{align*}
	The space of $\C$-curves in $\Ber$ is in bijection with the space of $\mathcal{D}$-surfaces in $\SO(5)$, which in turn is parametrized by $\SO(5)/\mathrm{T}^2$.
	
	We now show that each $\Gamma(p,E)$ is a $\C$-curve.  Let 
	\begin{equation*}
		L = \{ (\Sigma, p, \mathrm{o}) \in \Ber \times S^4 \times \mathbb{S} \left( \Lambda^2 \left( T^* S^4 \right) \right) \mid p \in \Sigma, \:\: \mathrm{o} \in \mathbb{S} \left( \Lambda^2 \left( T^*_p \Sigma \right) \right) \}.
	\end{equation*}
Note that $\SO(5)$ acts transitively on $L$ with stabiliser $\Sg^1_{(1,2)}$.  Letting $q_1 \colon L \to \Ber$ denote $q_1(\Sigma, p, \mathrm{o}) = \Sigma$ and $q_2 \colon L \to \text{Gr}_2^+(TS^4)$ denote $q_2(\Sigma, p, \mathrm{o}) = \left( T_p\Sigma, \mathrm{o} \right)$, we have a commutative diagram:
\begin{equation} \label{eq:DoubleFib}
\begin{tikzcd}
   & \SO(5) \arrow[d] \arrow[ldd, "\pi"'] \arrow[rdd, "\lambda"] &     \\
   & L \arrow[ld, "q_1"] \arrow[rd, "q_2"']   &     \\
\Ber &  & \text{Gr}_2^+(TS^4)
\end{tikzcd}
\end{equation}
For each $(p,E) \in \text{Gr}_2^+(TS^4)$, we now see that
\begin{align*}
\Gamma(p,E) = \left\{\Sigma \in \Ber \mid p \in \Sigma , \: T_p\Sigma = E\right\} & = (q_1 \circ q_2^{-1})(p,E) = (\pi \circ \lambda^{-1})(p,E).
\end{align*}
Thus, $\Gamma(p,E)$ is the $\pi$-image of the $\lambda$-fiber over $(p,E)$, hence is a $\mathcal{C}$-curve in $\Ber$.  By construction, the correspondence $(p,E) \mapsto \Gamma(p,E)$ is bijective. %The corresponding $\C$-curve is given by the collection of Veronese surfaces $\Sigma$ satisfying $f_1 \in \Sigma$ and $T_{f_1} \Sigma = \text{span} \left( f_2, f_3 \right).$
\end{proof}

%Let $\lambda : \SO(5) \to \SO(5)/\mathrm{T}^2$ be the map that sends the frame $\left( f_1, \ldots, f_5 \right)$ to the element $\left( f_1, \text{span}(f_2, f_3) \right)$ of  $\SO(5)/\mathrm{T}^2 \cong \text{Gr}_2 \left( T S^4 \right).$ 
%\begin{equation*}
%\begin{tikzcd}
%	& \mathrm{SO}(5) \arrow[ld, "\pi"'] \arrow[rd, "\lambda"] &                                                                    \\
%	\mathrm{SO}(5)/\mathrm{SO}(3) &   & \mathrm{SO}(5)/\mathrm{T}^2 \cong \text{Gr}_2 \!\left( T S^4 \right).
%\end{tikzcd}
%\end{equation*}

% (\textcolor{red}{Original}) To study associatives in $\Ber$ ruled by $\C$-curves, we consider the double fibration of $\SO(5)$ given in (\ref{eq:DoubleFib}).  By Proposition \ref{prop:Ccurves}, a ruled 3-submanifold $N$ of $\Ber$ gives a surface in $\text{Gr}_{2}\! \left( T S^4 \right)$, and conversely a surface $S$ in $\text{Gr}_{2}\! \left( T S^4 \right)$ gives a ruled 3-submanifold $(\pi \circ \lambda^{-1})(S)$ of $\Ber$.  The following theorem gives conditions on this surface that ensure the corresponding submanifold of $\Ber$ is associative. \\

 To study associatives in $\Ber$ ruled by $\C$-curves, we consider the double fibration of $\SO(5)$ given in (\ref{eq:DoubleFib}).  In view of Proposition \ref{prop:Ccurves}, we expect that if $S$ is a generic immersed surface in $\text{Gr}_{2}^+\! \left( T S^4 \right)$, then $(\pi \circ \lambda^{-1})(S)$ will be a ruled immersed 3-submanifold of $\Ber$.  We now clarify this point and fix notation.

Let $S \to \text{Gr}_{2}^+\! \left( T S^4 \right)$ be an immersed surface.  Recall the map $\lambda : \SO(5) \to \SO(5)/\mathrm{T}^2$ defined in (\ref{eq:lambda-map}). The structure equation (\ref{eq:SO5Struct}) gives
	\begin{align*}
	d \mathbf{e}_1 =& -\mu_{12} \mathbf{e}_2 - \mu_{13} \mathbf{e}_3 - \mu_{14} \mathbf{e}_4 - \mu_{15} \mathbf{e}_5, \\
	d \left( \mathbf{e}_2 \wedge \mathbf{e}_3 \right) =& - \mu_{13} \mathbf{e}_1 \wedge \mathbf{e}_2 + \mu_{12} \mathbf{e}_{1} \wedge \mathbf{e}_3 - \mu_{34} \mathbf{e}_2 \wedge \mathbf{e}_4 - \mu_{35} \mathbf{e}_2 \wedge \mathbf{e}_5 \\
	& + \mu_{24} \mathbf{e}_3 \wedge \mathbf{e}_4 + \mu_{25} \mathbf{e}_3 \wedge \mathbf{e}_5,
	\end{align*}
so the eight 1-forms that appear on the right-hand-side of these equations are $\lambda$-semibasic.  Let $g_{S}$ denote the restriction of the $\SO(5)$-invariant metric on $\text{Gr}_2^+(TS^4)$ given by % (\textcolor{red}{1. need to define $g_S$ properly})
	 \begin{equation*}
	\textstyle \frac{1}{2}\left( \mu_{12}^2 + \mu_{13}^2 + \mu_{14}^2 + \mu_{15}^2 + \mu_{24}^2 + \mu_{25}^2 + \mu_{34}^2 + \mu_{35}^2\right)
	 %\zeta_1 \circ \overline{\zeta}_1 + \zeta_2 \circ \overline{\zeta}_2 + \zeta_3 \circ \overline{\zeta}_3 + \zeta_4 \circ \overline{\zeta}_4
	 \end{equation*}
to $S$.  Let $\mathcal{F}(S)$ denote the $\Sg^1$-bundle of $g_{S}$-orthonormal frames on $S$, let $\mathcal{B}(S)$ denote the pullback of the $\mathrm{T}^2$-bundle $\SO(5) \to \text{Gr}_{2}^+\! \left( T S^4 \right)$ to $S$, and denote $\mathcal{F}(S) \times_S \mathcal{B}(S)$ by $\mathcal{G}(S).$ Then $\mathcal{G}(S)$ is a principal $\mathrm{T}^3$-bundle over $S$.
	 %	 Since $S$ is $J$-holomorphic, there exist $\mathbb{C}$-valued functions $W_1, W_2, W_3, W_4$ on $\mathcal{F}(S) \times_S \mathcal{B}(S)$ so that
%	 \begin{align*}
%	 \zeta_1 = W_1 \sigma, \:\: \zeta_2 = W_2 \sigma, \:\: \zeta_3 = W_3 \sigma, \:\: \zeta_4 = W_4 \sigma, \\
%	 |W_1|^2 + |W_2|^2 + |W_3|^2 + |W_4|^2 = 1,
%	 \end{align*}
%	 on $\mathcal{F}(S) \times_S \mathcal{B}(S).$ Note that the functions $\lvert W_1 \rvert , \ldots, \lvert W_4 \rvert$ descend to well-defined functions on $S.$

It is straightforward to check that the image of natural map $\mathcal{G}(S) \to \Ber$ is exactly $(\pi \circ \lambda^{-1})(S)$, and is a 3-dimensional $\mathcal{C}$-ruled submanifold on the locus of points $p$ where the linear map $\left( \omega_{{1}}, \ldots, \omega_{{7}} \right) : T_p \mathcal{G}(S) \to \R^7$ is injective.
\begin{equation*}
\begin{tikzcd}
\mathcal{G}(S) \arrow[d] \arrow[r] & \mathcal{B}(S) \arrow[d] \arrow[r] & \text{SO}(5) \arrow[d, "\lambda"'] \arrow[rd, "\pi"] &                           \\
S \arrow[r]                                                & S \arrow[r]                        & \text{Gr}_2^+(TS^4)                                  & \Ber
\end{tikzcd}
\end{equation*}

We now describe a sort of inverse to this construction.  Let $N \to \Ber$ be a ruled $3$-submanifold of $\Ber$.  Let $\mathcal{F}_{\text{ON}}(N)$ be the orthonormal coframe bundle of $N$ with respect to the induced metric on $N,$ and denote the tautological $\R^3$-valued 1-form on $\mathcal{F}_{\text{ON}}(N)$ by $(\alpha_1, \alpha_2, \alpha_3)$. Let $\mathcal{B}(N)$ denote the pullback of the $\SO(3)$-bundle $\pi: \SO(5) \to \Ber$ to $N$. 

Since $N$ is ruled, there is a $\C$-curve through each point $\Sigma \in N$ that remains in $N$. Let $V$ denote the unit vector field on $N$ that is tangent to the $\C$-ruling. Define a subbundle $\mathcal{F}'(N)$ of $\mathcal{F}_{\text{ON}}(N) \times_N \mathcal{B}(N)$ by letting the fibre over $\Sigma \in N$ be
	\begin{multline*}
	\mathcal{F}'(N)_{\Sigma} = \left\lbrace \left( u, \left( f_1, \ldots, f_5 \right) \right) \in \mathcal{F}_{\text{ON}}(N)_{\Sigma} \times \pi^{-1} \left( \Sigma \right) \mid u_* \left( V \right) = \left(1, 0, 0 \right),\: \right. \\ \left. \text{the ruling curve through $\Sigma$ is given by} \:\: \Gamma\!\left( f_1, \text{span}(f_2, f_3) \right) \right\rbrace.
	\end{multline*}
	The bundle $\mathcal{F}'(N)$ is a principal $\mathrm{T}^2$-bundle over $N$.  One can check that the image of the natural map $\mathcal{F}'(N) \to \SO(5)$ is the $\lambda$-preimage of a surface $S \subset \text{Gr}_2^+(TS^4)$, and hence $N = (\pi \circ \lambda^{-1})(S)$.
\begin{equation*}
\begin{tikzcd}
\mathcal{F}'(N) \arrow[r, hook] \arrow[d] & \mathcal{F}_{\text{ON}}(N) \times_N \mathcal{B}(N) \arrow[d] \arrow[r] & \mathcal{B}(N) \arrow[d] \arrow[r] & \text{SO}(5) \arrow[d, "\pi"'] \arrow[rd, "\lambda"] &                     \\
N \arrow[r, no head]                      & N \arrow[r, no head]                                                   & N \arrow[r]                        & \Ber                            & \text{Gr}_2^+(TS^4)
\end{tikzcd}
\end{equation*}
	
The following theorem gives conditions on the surface $S \to \text{Gr}_2^+(TS^4)$ that ensure the corresponding ruled $3$-submanifold $N \to \Ber$ is associative.

\begin{thm}\label{thm:RuledAssoc}
	There is an $\SO(5)$-invariant (non-integrable) almost complex structure $J$ on $\mathrm{Gr}_2^+(TS^4)$ such that:
	\begin{enumerate}
		\item Any ruled associative submanifold of $\Ber$ is locally the $\pi \circ \lambda^{-1}$-image of a $J$-holomorphic curve $\gamma : S \to \mathrm{Gr}_2^+(TS^4)$. 
		\item For each $J$-holomorphic curve $\gamma : S \to \mathrm{Gr}_2^+(TS^4)$ not locally equivalent to the Gauss lift of a Veronese surface, there is a dense subset $S^{\circ} \subset S$ such that $\pi \circ \lambda^{-1} \circ \gamma \left( S^{\circ} \right)$ is a $\C$-ruled associative submanifold of $\Ber.$
	\end{enumerate}
\end{thm}

\begin{proof}
	We begin by defining the almost complex structure $J$ on $\text{Gr}_2^+(TS^4)$. Define $\mathbb{C}$-valued 1-forms on $\SO(5)$ by
%	Consider the map $\lambda : \SO(5) \to \SO(5)/\mathrm{T}^2$ defined above. The structure equation (\ref{eq:SO5Struct}) gives
%	\begin{align*}
%	d \mathbf{e}_1 =& -\mu_{12} \mathbf{e}_2 - \mu_{13} \mathbf{e}_3 - \mu_{14} \mathbf{e}_4 - \mu_{15} \mathbf{e}_5, \\
%	d \left( \mathbf{e}_2 \wedge \mathbf{e}_3 \right) =& - \mu_{13} \mathbf{e}_1 \wedge \mathbf{e}_2 + \mu_{12} \mathbf{e}_{1} \wedge \mathbf{e}_3 - \mu_{34} \mathbf{e}_2 \wedge \mathbf{e}_4 - \mu_{35} \mathbf{e}_2 \wedge \mathbf{e}_5 \\
%	& + \mu_{24} \mathbf{e}_3 \wedge \mathbf{e}_4 + \mu_{25} \mathbf{e}_3 \wedge \mathbf{e}_5,
%	\end{align*}
%	so the eight 1-forms that appear on the right-hand-side of these equations are $\lambda$-semibasic.  % The fibres of $\lambda$ are diffeomorphic to $\mathrm{T}^2$. In particular, they are connected.
	\begin{align*}
	\zeta_1 & = \textstyle \frac{1}{\sqrt{2}} \left( \mu_{12} - i \mu_{13} \right), & \zeta_3 & \textstyle = \frac{1}{2}\left[ (\mu_{24}- \mu_{35}) + i ( \mu_{25} + \mu_{34} ) \right], \\
	 \zeta_2 & = \textstyle \frac{1}{\sqrt{2}} \left( \mu_{14} - i \mu_{15} \right), & \zeta_4 & \textstyle = \frac{1}{2}\left[ (\mu_{24} + \mu_{35}) - i ( \mu_{25} - \mu_{34} ) \right].
	\end{align*}
	These forms are $\lambda$-semibasic and satisfy the equations
	\begin{equation}\label{eq:Jstruct} \left.
	\begin{aligned}
	d \zeta_1 &\equiv \overline{\zeta}_2 \wedge \overline{\zeta}_3 \\
	d \zeta_2 &\equiv \overline{\zeta}_3 \wedge \overline{\zeta}_1 \\
	d \zeta_3 &\equiv \overline{\zeta}_1 \wedge \overline{\zeta}_2 \\
	d \zeta_4 &\equiv 0
	\end{aligned} \:\:\: \right\rbrace \:\: \text{mod} \:\:\: \zeta_1, \zeta_2, \zeta_3, \zeta_4.
	\end{equation}
	Since the fibres of $\lambda$ are connected, it follows that there is a unique almost complex structure $J$ on $\SO(5)/\mathrm{T}^2$ so that the $(1,0)$-forms on $\SO(5)/\mathrm{T}^2$ pull back to be linear combinations of $\zeta_1, \zeta_2, \zeta_3, \zeta_4.$ Equations (\ref{eq:Jstruct}) imply that $J$ is non-integrable.  For future use, we remark that
\begin{equation} \label{eq:omegazeta}
\begin{bmatrix}
\omega_2 + i\omega_3 \\
\omega_4 + i\omega_5 \\
\omega_6 - i\omega_7 \\
\gamma_2 - i\gamma_3
\end{bmatrix} = \frac{1}{10}\begin{pmatrix}
-6 & 0 & 0 & 3\sqrt{6} \\
0 & -3\sqrt{10} & 0 & 0 \\
0 & 0 & -3\sqrt{10} & 0 \\
2\sqrt{6}\,i & 0 & 0 & 4i
\end{pmatrix}\begin{bmatrix}
\zeta_1 \\
\zeta_2 \\
\zeta_3 \\
\zeta_4
\end{bmatrix}.
\end{equation}
	
%	We now prove statement (1). Let $N \subset \Ber$ be a ruled associative submanifold. %and let $\mathcal{B}(N)$ denote the pullback of the $\SO(3)$-subbundle $\pi: \SO(5) \to \Ber$ to $N$. Let $\mathcal{F}_{\text{ON}}(N)$ be the orthonormal coframe bundle of $N$ with respect to the induced metric on $N,$ and denote the tautological $\R^3$-valued 1-form on $\mathcal{F}_{\text{ON}}(N)$ by $(\alpha_1, \alpha_2, \alpha_3).$ Since $N$ is ruled, there is a $\C$-curve through each point $\Sigma \in N$ that remains in $N.$ Let $V$ denote the unit vector field on $N$ that is tangent to the $\C$-ruling. Define a subbundle $\mathcal{F}'(N)$ of $\mathcal{F}_{\text{ON}}(N) \times_N \mathcal{B}(N)$ by letting the fibre over $\Sigma \in N$ be
%	\begin{multline*}
%	\mathcal{F}'(N)_{\Sigma} = \left\lbrace \left( u, \left( f_1, \ldots, f_5 \right) \right) \in \mathcal{F}_{\text{ON}}(N)_{\Sigma} \times \pi^{-1} \left( \Sigma \right) \mid u_* \left( V \right) = \left(1, 0, 0 \right),\: \right. \\ \left. \text{the ruling curve through $\Sigma$ is given by} \:\: \Gamma \left( f_1, \text{span}(f_2, f_3) \right) \right\rbrace.
%	\end{multline*}
%	The bundle $\mathcal{F}'(N)$ is a principal $\Sg^1$-bundle over $N.$
	
	We now prove statement (1). Let $N \to \Ber$ be a ruled associative submanifold.  As above, we let $V$ denote the unit vector field on $N$ that is tangent to the $\C$-ruling, and perform computations on $\mathcal{F}'(N)$.  Since $V$ is tangent to a $\C$-curve, we have $\omega_{1} = \alpha_1$ and
	\begin{align*}
	d \mathbf{e}_1 \equiv 0 \:\:\: \text{mod} \:\:\: \alpha_2, \alpha_3, \\
	d \left( \mathbf{e}_2 \wedge \mathbf{e}_3 \right) \equiv 0 \:\:\: \text{mod} \:\:\: \alpha_2, \alpha_3.
	\end{align*}
	on $\mathcal{F}'(N).$ Let $\sigma = \alpha_2 + i \alpha_3.$ It follows from (\ref{eq:SO5Struct}) that there exist $\mathbb{C}$-valued functions $A, B, C, F, G, X, Y, Z$ on $\mathcal{F}'(N)$ so that
	\begin{align*}
	\omega_2 + i \omega_3 & = A \sigma + X \overline{\sigma}, & \gamma_2 - i \gamma_3 & = F \sigma + G \overline{\sigma}, \\
	\omega_4 + i \omega_5 & = B \sigma + Y \overline{\sigma}, \\
	\omega_6 - i \omega_7 & = C \sigma + Z \overline{\sigma},
	\end{align*}
	 on $\mathcal{F}'(N)$. Pulling back to $\mathcal{F}'(N)$ we have
	 \begin{align*}
	 \varphi |_{N} &= \frac{i}{2} \left( \left\lvert A \right\rvert^2 + \left\lvert B \right\rvert^2 + \left\lvert C \right\rvert^2 - \left\lvert X \right\rvert^2 - \left\lvert Y \right\rvert^2 - \left\lvert Z \right\rvert^2 \right) \alpha_1 \wedge \sigma \wedge \overline{\sigma}, \\
	 g |_{N} &= \alpha_1^2 + \Re \left( \left( A \overline{X} + B \overline{Y} + C \overline{Z} \right) \sigma^2 \right) + \left( \left\lvert A \right\rvert^2 + \left\lvert B \right\rvert^2 + \left\lvert C \right\rvert^2 + \left\lvert X \right\rvert^2 + \left\lvert Y \right\rvert^2 + \left\lvert Z \right\rvert^2 \right) \sigma \circ \overline{\sigma} \\
	 \text{vol}_N & = \frac{i}{2} \left( \left( \left\lvert A \right\rvert^2 + \left\lvert B \right\rvert^2 + \left\lvert C \right\rvert^2 + \left\lvert X \right\rvert^2 + \left\lvert Y \right\rvert^2 + \left\lvert Z \right\rvert^2 \right)^2 - 4 \left\lvert A \overline{X} + B \overline{Y} + C \overline{Z} \right\rvert^2 \right)^{1/2} \alpha_1 \wedge \sigma \wedge \overline{\sigma},
	 \end{align*}
	 
	 Now, $N$ is associative, so $\varphi |_N = \text{vol}_N.$ It follows from the above formulas, and an application of the Cauchy-Schwarz inequality, that the vectors $\left(A, B, C \right), \left(X, Y, Z\right) \in \mathbb{C}^3$ are parallel. That is, there are $\mathbb{C}$-valued functions $u$ and $v$ on $\mathcal{F}'(N)$ so that $u X = v A,  u Y= v B, u Z = v C.$ So, the following equations hold on $\mathcal{F}'(N)$:
	 \begin{align*}
	 \left( \omega_2 + i \omega_3 \right) \wedge \left( \omega_4 + i \omega_5 \right) = 0, \\
	 \left( \omega_4 + i \omega_5 \right) \wedge \left( \omega_6 - i \omega_7 \right) = 0, \\
	 \left( \omega_6 - i \omega_7 \right) \wedge \left( \omega_2 + i \omega_3 \right) = 0.
	 \end{align*}
	 Differentiating these equations using the structure equations (\ref{eq:BerStruct1}) we find
	 \begin{align*}
	 B \left( u G - v F \right) \, \alpha_1 \wedge \sigma \wedge \overline{\sigma} = 0, \:\:\:\:\:
	 C \left( u G - v F \right) \, \alpha_1 \wedge \sigma \wedge \overline{\sigma} = 0,
	 \end{align*}
	 so the functions $B \left(u G - v F \right)$ and $C \left( u G -v F \right)$ vanish identically on $\mathcal{F}'(N).$
	 If the function $u G - v F$ vanishes identically on $\mathcal{F}'(N),$ then $\left( \omega_2 + i \omega_3 \right) \wedge \left( \gamma_2 - i \gamma_3 \right) = 0$ on $\mathcal{F}'(N).$ If $G$ does not vanish identically on $\mathcal{F}'(N),$ we may restrict to the dense open set on which $B = C = 0.$ On this set we have $\omega_4 = \omega_5 = \omega_6 = \omega_7 = 0.$ The structure equations (\ref{eq:BerStruct1}) imply that in this case too we have
	 \begin{equation*}
	  \left( \omega_2 + i \omega_3 \right) \wedge \left( \gamma_2 - i \gamma_3 \right) = 0.
	 \end{equation*}
	 
	 By (\ref{eq:omegazeta}), on $\SO(5)$ we have
	 \begin{equation*}
	 \begin{aligned}
	 \omega_2 + i \omega_3 \equiv \omega_4 + i \omega_5 \equiv \omega_6 - i \omega_7 \equiv \gamma_2 - i \gamma_3 \equiv 0 \:\:\: \text{mod} \:\:\: \zeta_1, \zeta_2, \zeta_3, \zeta_4.
	 \end{aligned}
	 \end{equation*}
	 It follows that the image of the natural map $\mathcal{F}'(N) \to \text{Gr}_2^+(TS^4)$ is a $J$-holomorphic curve. This proves statement (1).

 %	 (\textcolor{red}{this paragraph is copied-pasted above}) Let $g_{S}$ denote the restriction of the $\SO(5)$-invariant metric
%	 \begin{equation*}
%	 \zeta_1 \circ \overline{\zeta}_1 + \zeta_2 \circ \overline{\zeta}_2 + \zeta_3 \circ \overline{\zeta}_3 + \zeta_4 \circ \overline{\zeta}_4
%	 \end{equation*}
%	 to $S$. Let $\mathcal{F}(S)$ denote the $\mathrm{U}(1)$-bundle of $g_{S}$-orthonormal frames on $S,$ with $\mathbb{C}$-valued tautological form $\sigma$ and $i \mathfrak{u}(1)$-valued connection form $i \kappa.$ Let $\mathcal{B}(S)$ denote the pullback of the $\mathrm{U}(1)^2$-bundle $\SO(5) \to \SO(5)/\mathrm{T}^2$ to $S.$ Then $\mathcal{F}(S) \times_S \mathcal{B}(S)$ is a principal $\mathrm{U}(1)^3$ bundle over $S$. \\

	We now prove statement (2). Let $S \to \text{Gr}_2^+(TS^4)$ be a $J$-holomorphic curve. Since $S$ is $J$-holomorphic, there exist $\mathbb{C}$-valued functions $W_1, W_2, W_3, W_4$ on $\mathcal{G}(S)$ so that
	 \begin{align} \label{eq:WDef}
	 \zeta_1 = W_1 \sigma, \:\: \zeta_2 = W_2 \sigma, \:\: \zeta_3 = W_3 \sigma, \:\: \zeta_4 = W_4 \sigma, \\
	 |W_1|^2 + |W_2|^2 + |W_3|^2 + |W_4|^2 = 1, \label{eq:Wsqeq}
	 \end{align}
	 on $\mathcal{G}(S)$. Using (\ref{eq:omegazeta}) and (\ref{eq:WDef}), one can check that the natural map $\mathcal{G}(S) \to \Ber$ is a $\mathcal{C}$-ruled associative immersion on the locus where the functions $2 W_1 -\sqrt{6} W_4, W_2, W_3$ do not simultaneously vanish. %Note that the functions $\lvert W_1 \rvert , \ldots, \lvert W_4 \rvert$ descend to well-defined functions on $S.$
 
	 Now, the functions $|W_1|^2, |W_2|^2, |W_3|^2, |W_4|^2$ on $\mathcal{G}(S)$ descend to well-defined functions on $S$, and we may consider the locus
	 \begin{equation*}
	 Z = \{ 2|W_1|^2 = 3|W_4|^2 \text{ and } |W_2|^2 = 0 \text{ and } |W_3|^2 = 0\} \subset S.
	 \end{equation*}
	 With respect to the complex structure on $S$ induced by the $J$-holomorphic map $S \to \text{Gr}_2^+(TS^4)$, one can compute that
	 \begin{align*}
	 \overline{\partial} W_1 & = W_2\overline{W}_4\,\overline{\sigma} & \overline{\partial} W_3 & = 0 \\
	 \overline{\partial} W_2 & = 0 & \overline{\partial} W_4 & = \overline{W}_1W_2\,\overline{\sigma}
	 \end{align*}
	 In particular, $W_2$ and $W_3$ are holomorphic.  If $W_2$  is not identically zero, then its zero set is discrete, and hence $Z$ is discrete. The same reasoning applies to $W_3.$ If both $W_2$ and $W_3$ are identically zero, then $W_1$ and $W_4$ are holomorphic, and the set $Z$ is the vanishing locus of a real analytic function. Defining $S^\circ = S \setminus Z$, we see that the open set $S^\circ$ is either dense or empty. 

If $S^\circ$ is empty, we have $2|W_1|^2 = 3|W_4|^2$ and $|W_2| = |W_3| = 0$ on $S,$ and hence by (\ref{eq:Wsqeq}), $\left\lvert W_1 \right\rvert^2 = 3/5$ and $\left\lvert W_4 \right\rvert^2 = 2/5$ on $S.$ Define a subbundle $\mathcal{G}'(S) \subset \mathcal{G}(S) $ by the condition $W_1 = \sqrt{3 / 5}, W_4 = \sqrt{2 / 5}.$ By (\ref{eq:omegazeta}), on this subbundle, $\omega_{{1}} = \ldots = \omega_{{7}} = 0,$ and the Maurer-Cartan form of $\SO(5)$ restricts to be
\begin{small}
	\begin{equation*}
	\mu = \left[ \begin {array}{ccccc} 0 & -\sqrt {3}\gamma_{{3}} & \sqrt {3}\gamma_{{2}} & 0 & 0 \\
	\sqrt {3}\gamma_{{3}} & 0 & -\gamma_{{1}} & -\gamma_{{3}} & \gamma_{{2}} \\
	-\sqrt {3}\gamma_{{2}} & \gamma_{{1}} & 0 & -\gamma_{{2}} & -\gamma_{{3}} \\
	0 & \gamma_{{3}} & \gamma_{{2}} & 0 &-2\gamma_{{1}} \\
	0 & -\gamma_{{2}} & \gamma_{{3}} & 2\gamma_{{1}} & 0
	\end {array} \right].
	\end{equation*}
\end{small}
Thus,  $S$ is the Gauss lift of a Veronese surface to $\gra.$ This proves statement (2).
%and hence by (\ref{eq:omegazeta}) and (\ref{eq:WDef}), ...... and hence the map $S \to \text{Gr}_2^+(TS^4)$ is the Gauss lift of a Veronese surface $S \to S^4$.  This proves statement (2).
%Consequently, $W_2$ and $W_3$ are holomorphic, so the locus $\{|W_2|^2 = 0 \text{ and} |W_3| = 0\}$ is either a discrete set or all of $S$.  In the latter case, $W_2 = 0$  ...... Either $Z$ is a discrete set or $Z = S$.	 
%	 Let $S^\circ$ be the subset of $S$ where $2 \lvert W_1 \rvert^2 \neq 3\lvert W_4 \rvert^2, \lvert W_2 \rvert^2 \neq 0, \lvert W_3 \rvert^2 \neq 0$. It is straightforward to check that $S^\circ$ is a dense subset of $S$ unless $S$ is the Gauss lift of a Veronese surface to $\SO(5)/\mathrm{T}^2.$ This proves statement (2).
\end{proof}

\begin{cor}
	Ruled associative submanifolds of $\Ber$ exist locally and depend on 6 functions of 1 variable.
\end{cor}

\begin{remark}
	Theorem \ref{thm:RuledAssoc} may be thought of as an analogue of theorems of Bryant \cite{Bry06} and Fox \cites{Fox2007ConesPlanes,Fox2008Cayley} characterising the ruled submanifolds under consideration as pseudo-holomorphic curves in the respective spaces of rulings.
\end{remark}

\begin{remark}
 For a $J$-holomorphic curve in $\text{Gr}_2^+(TS^4)$ we may think of the locus $S  \setminus S^\circ$ as the set of points where $S$ osculates to second order with the lift of a Veronese surface. In light of the fact that the $\mathcal{C}$-ruled associative submanifold associated to $S$ is defined as the set of Veronese surfaces sharing first order contact with $S,$ it is not surprising that issues occur at  $S  \setminus S^\circ.$ A similar phenomenon occurs in the theory of curves in the plane: if a plane curve has a vertex (a point where the derivative of its curvature function is zero), then its evolute (the curve traced out by the centres of the osculating circles) has a cusp.
 
From the proof of Theorem \ref{thm:RuledAssoc}, the ruled associative corresponding to $S^\circ$ is given by the image of the natural map $\mathcal{G}(S^{\circ}) \to \Ber$. However, the locus on $\mathcal{G}(S)$ where the functions $2 W_1 -\sqrt{6} W_4, W_2, W_3$ do not simultaneously vanish is strictly larger than $\mathcal{G}(S^\circ)$: if $p \in S  \setminus S^\circ$ then the locus in the fibre $\mathcal{G}(S)_p$ with $2 W_1 -\sqrt{6} W_4 = 0, W_2 = 0, W_3 = 0$ has positive codimension. Thus, the image of the map $\mathcal{G}(S) \setminus \left\lbrace 2 W_1 -\sqrt{6} W_4 = 0, W_2 = 0, W_3 = 0 \right\rbrace \to \Ber$ is an associative submanifold extending the ruled associative corresponding to $S^\circ.$
\end{remark}

We now proceed to study the $J$-holomorphic curves in $\text{Gr}_2^+(TS^4)$. The structure equations of $\text{Gr}_2^+(TS^4)$ written in terms of the basis $\zeta_i$ defined in the proof of Theorem \ref{thm:RuledAssoc} are
\begin{equation}\label{eq:StructFlag}
\begin{aligned}
d \zeta_1 & = -i \left( \rho_1 - \rho_2 \right) \wedge \zeta_1 + \zeta_2 \wedge \overline{\zeta_4} + \overline{\zeta_2} \wedge \overline{\zeta_3}, & d \rho_1 & = -\tfrac{i}{2} \left(\zeta_1 \wedge \overline{\zeta_1} + \zeta_2 \wedge \overline{\zeta_2} -2 \zeta_3 \wedge \overline{\zeta_3}\right), \\
d \zeta_2 & = -i \left( \rho_1 + \rho_2 \right) \wedge \zeta_2 - \zeta_1 \wedge {\zeta_4} + \overline{\zeta_3} \wedge \overline{\zeta_1}, & d \rho_2 & = -\tfrac{i}{2} \left( - \zeta_1 \wedge \overline{\zeta_1} + \zeta_2 \wedge \overline{\zeta_2} + 2 \zeta_4 \wedge \overline{\zeta_4}\right), \\
d \zeta_3 & = 2 i \rho_1 \wedge \zeta_3 + \overline{\zeta_1} \wedge \overline{\zeta_2}, & & \\
d \zeta_4 & = -2 i \rho_2 \wedge \zeta_4 + \overline{\zeta_1} \wedge {\zeta_2}, & &
\end{aligned}
\end{equation}
where $\rho_1 = \tfrac{1}{2} \left(\mu_{23} + \mu_{45} \right), \rho_2 = -\tfrac{1}{2} \left(\mu_{23} - \mu_{45} \right)$ are connection forms for the $\mathrm{T}^2$-structure on $\SO(5)/\mathrm{T}^2.$

\begin{remark}
	The space $\SO(5)/\mathrm{T}^2 \cong \mathrm{Gr}_2^+\! \left( T S^4 \right)$ carries two tautological $\R^2$-bundles, corresponding to the projections
	\begin{align*}
	& \mathrm{Gr}_2^+\! \left( T S^4 \right) \ni \left( p, E \right) \mapsto E \in \mathrm{Gr}_2 (\R^5), \\
	& \mathrm{Gr}_2^+\! \left( T S^4 \right) \ni \left( p, E \right) \mapsto \left( \textrm{span}(p) \oplus E \right)^\perp \in \mathrm{Gr}_2 (\R^5).
	\end{align*}
	Let $P_1$ and $P_2$ denote the associated circle bundles. The forms $\rho_1 + \rho_2, \rho_1 - \rho_2$ are connection forms for $P_1, P_2$ respectively. If $S \to \mathrm{Gr}_2^+\! \left( T S^4 \right)$ is a $J$-holomorphic curve, then one sees from the proof of Theorem \ref{thm:RuledAssoc} that the $\mathcal{C}$-ruled associative submanifold $N$ corresponding to $S$ is topologically the total space of the circle subbundle of the pullback of $P_1 \times_{\text{Gr}_2^+(TS^4)} P_2 \to \text{Gr}_2^+(TS^4)$ to $S$ with connection form $-\rho_1 + 3 \rho_2.$ 
\end{remark}

It follows from the structure equations (\ref{eq:StructFlag}) that the subspace of $\mathfrak{so}(5)$ spanned by the dual vectors to $\zeta_4, \overline{\zeta_4}, \rho_1, \rho_2$ defines a subalgebra isomorphic to $\mathfrak{u}(2),$ which exponentiates to a subgroup $\mathrm{U}(2) \subset \SO(5)$.  The forms $\zeta_1, \zeta_2, \zeta_3$ are semi-basic for the projection $\SO(5) \to \SO(5)/\mathrm{U}(2)$. From the structure equations (\ref{eq:StructFlag}) we see that the forms
\begin{equation*}
\tfrac{i}{2} \left( \zeta_1 \wedge \overline{\zeta_1} + \zeta_2 \wedge \overline{\zeta_2} + \zeta_3 \wedge \overline{\zeta_3} \right) \:\:\: \text{and} \:\:\: \zeta_1 \wedge \zeta_2 \wedge \zeta_3
\end{equation*}
define an $\SO(5)$-invariant \emph{nearly K\"ahler} structure on $\SO(5)/\mathrm{U}(2).$ In fact, $\SO(5)/\mathrm{U}(2)$ is diffeomorphic to $\mathbb{CP}^3,$ and the nearly K\"ahler structure just described is exactly the well-known homogeneous nearly K\"ahler structure arising from the twistor fibration $\mathbb{CP}^3 \to S^4$. Let $J_{\text{NK}}$ denote the almost complex structure associated to this nearly K\"ahler structure on $\SO(5)/\mathrm{U}(2).$

The $(1,0)$-forms on $(\mathbb{CP}^3, J_{\text{NK}})$ are exactly the 1-forms whose pullbacks to $\SO(5)$ are linear combinations of $\zeta_1, \zeta_2, \zeta_3$. It follows that the projection of a $J$-holomorphic curve in $\mathrm{Gr}_2^+\! \left( T S^4 \right)$ is a $J_{\text{NK}}$-holomorphic curve in $\mathbb{CP}^3.$ The next proposition provides a converse to this construction.

\begin{prop}\label{prop:JNK}
	Any $J_{\text{NK}}$-holomorphic curve in $\mathbb{CP}^3$ has a unique lift to a $J$-holomorphic curve in $\mathrm{Gr}_2^+\! \left( T S^4 \right)$ satisfying $\zeta_2 = 0.$
\end{prop}

\begin{proof}
	The isotropy representation of $\mathbb{CP}^3 \cong \SO(5)/\mathrm{U}(2)$ splits the cotangent space $T^* \mathbb{CP}^3$ into a direct sum of bundles
	\begin{equation*}
	T^* \mathbb{CP}^3 = \mathcal{H}^* \oplus \mathcal{V}^*,
	\end{equation*}
	where the pullback of $\mathcal{H}^*$ to $\SO(5)$ is spanned by $\zeta_1, \zeta_2,$ and the pullback of $\mathcal{V}^*$ is spanned by $\zeta_3.$ The bundle $\mathcal{H}$ is modeled on the $\mathrm{U}(2)$-representation $\mathbb{C}^2$. 
	
	Let $S$ be a $J_{\text{NK}}$-holomorphic curve in $\mathbb{CP}^3.$ Since  $\mathrm{U}(2)$ acts transitively on the complex lines in $\mathbb{C}^2,$ we may adapt frames so that $\zeta_2=0$ on $S$. Let $\mathcal{E}(S)$ denote the bundle of frames adapted in this way. The structure equations (\ref{eq:StructFlag}) then imply that $\zeta_1 \wedge \zeta_4 = 0$ on $\mathcal{E}(S)$.  The map $\mathcal{E}(S) \to \mathrm{Gr}_2^+\! \left( T S^4 \right)$ given by composing the natural map $\mathcal{E}(S) \to \SO(5)$ with the coset projection $\lambda: \SO(5) \to \mathrm{Gr}_2^+\! \left( T S^4 \right)$ does not depend on the choice of coframe, so provides a $J$-holomorphic lift $S \to \mathrm{Gr}_2^+\! \left( T S^4 \right)$.
\end{proof}

The $J_{\text{NK}}$-holomorphic curves in $\mathbb{CP}^3$ have been studied by Xu \cite{Xu10}. There are two classes of curves of particular interest. A  $J_{\text{NK}}$-holomorphic curve in $\mathbb{CP}^3$ is called \emph{horizontal} if it is horizontal for the twistor projection $\mathbb{CP}^3 \to \mathbb{HP}^1 \cong S^4,$ i.e. if $\zeta_3 = 0$ on the curve. A $J_{\text{NK}}$-holomorphic curve in $\mathbb{CP}^3$ is called \emph{null-torsion} if, after adapting frames so that $\zeta_2=0$, one has $\zeta_4 = 0$ on the adapted coframe bundle.

By work of Bryant \cite{Bry82}, the horizontal curves in $\mathbb{CP}^3$ are the twistor lifts of superminimal surfaces in $S^4$ of positive spin. Xu shows that the the null-torsion curves in $\mathbb{CP}^3$ are the lifts of superminimal surfaces in $S^4$ of negative spin (note that the antipodal map in $S^4$ maps superminimal surfaces of positive spin to ones of negative spin and vice versa). In both cases the resulting $J$-holomorphic curve in $\gra$ is simply the Gauss lift of the original surface in $S^4.$ Superminimal surfaces in $S^4$ admit a Weierstrass formula \cite{Bry82}, and using this formula it is possible in principle to write down a formula giving the associated ruled associative submanifolds of $\Ber$ for both of these cases.

\begin{thm}\label{thm:TopType}
	There exist infinitely many topological types of compact (immersed, generically 1-1) $\mathcal{C}$-ruled associative submanifolds of $\Ber.$
\end{thm}

\begin{proof}
	Bryant has proven \cite{Bry82} that every compact Riemann surface may be conformally immersed as a superminimal surface in $S^4$ which is generically 1-1. By the work of Xu \cite{Xu10} described above, any superminimal surface of negative spin in $S^4$ gives rise to a null-torsion $J_{\text{NK}}$-holomorphic curve in $\mathbb{CP}^3,$ and from Proposition \ref{prop:JNK} this further lifts to a $J$-holomorphic curve $S$ in $\SO(5)/\mathrm{T}^2$.
	
	On such a lift $S$ we have $\zeta_2 = \zeta_4 = 0,$ but $\zeta_1 \neq 0,$ since the induced metric on the superminimal surface in $S^4$ is given by $\zeta_1 \circ \overline{\zeta_1}.$ Thus we have $S^\circ = S,$ and Theorem \ref{thm:RuledAssoc} associates to $S$ a $\mathcal{C}$-ruled associative submanifold $N$ of $\Ber.$ The associative submanifold $N$ is topologically a circle bundle over $S,$ and so Bryant's result provides infinitely many topological types of such associative submanifolds.
\end{proof}

\begin{remark}
  If one begins with an embedded superminimal surface in $S^4$ then the lift $S \to \mathrm{Gr}_2^+\! \left( T S^4 \right)$ described in the proof will also be embedded. The non-embedded points of the associative submanifold $N$ can then only arise from intersections in $\mathrm{Gr}_2^+\! \left( T S^4 \right)$ of the lift $S$ with lifts of the Veronese surfaces comprising $N$ and one expects that in the generic case the set of such intersections is empty.
\end{remark}

%%%%%%%%%%%%%%%%%%%%%%%%%%%%%%%%%%%%%%%%%%%%%%%%%%%%%%%%%%
\section{Associative submanifolds with special Gauss map}\label{sect:Gauss}
%%%%%%%%%%%%%%%%%%%%%%%%%%%%%%%%%%%%%%%%%%%%%%%%%%%%%%%%%%

 Let $N \to \Ber$ be an associative submanifold.  The \emph{Gauss map} of $N$ is the map
\begin{align*}
N & \to \text{Gr}_{\textup{ass}}\! \left( \Ber \right) \\
p & \mapsto T_p N
\end{align*} 
where $\text{Gr}_{\textup{ass}} \left( \Ber \right)$ denotes the Grassmann bundle of associative 3-planes over $\Ber$.  The fibres of $\text{Gr}_{\textup{ass}}\! \left( \Ber \right) \to \Ber$ are diffeomorphic to the $8$-dimensional homogeneous space $\G_2 / \SO(4) \cong \text{Gr}_{\textup{ass}}\! \left( \R^7 \right)$.

Since $\Ber$ is a homogeneous space, we may translate each tangent plane $T_pN$ by an ambient motion $g \in \SO(5)$ to lie in a fixed $T_q\Ber \simeq \mathbb{R}^7$, and thereby obtain a well-defined map %the image of the Gauss map to obtain a well-defined map
\begin{align*}
N & \to \text{Gr}_{\textup{ass}}\! \left( \R^7 \right) / \SO(3) \\
p & \mapsto \SO(3) \cdot [g_*^{-1}T_p N]
\end{align*}
Here, $\text{Gr}_{\textup{ass}} \!\left( \R^7 \right) / \SO(3)$ denotes the orbit space for the action of $\SO(3)$ on $\text{Gr}_{\textup{ass}}\! \left( \R^7 \right)$ that is induced by the $\SO(3)$-action on $\mathbb{R}^7$.  One difficulty in the study of associative submanifolds of the Berger space is the complicated nature of the $\SO(3)$-action on $\text{Gr}_{\text{ass}}(\mathbb{R}^7)$.  While the generic $\text{SO}(3)$-orbit has trivial stabiliser, there also exist both singular orbits (whose stabiliser is a continuous group) and exceptional orbits (whose stabiliser is a finite group).  %A fundamental difficulty in the study of associative submanifolds of the Berger space is the complicated nature of the orbit space $\text{Gr}_{\textup{ass}}\! \left( \R^7 \right) / \SO(3)$.

In this section, we shall classify the associative submanifolds of $\Ber$ with \emph{Gauss map of special type}, i.e., those associative submanifolds $N \subset \Ber$ all of whose tangent spaces $T_p N$ have non-trivial $\SO (3)$-stabiliser.  To that end, the main technical tool we will use is the following classical result, known as \textit{Cartan's Theorem on Maps into Lie Groups}:

 \begin{thm} \label{thm:MaurerCartan} \cites{IvLaSecond, EDSBook} Let $\G$ be a Lie group with Lie algebra $\mathfrak{g}$, and let $\omega \in \Omega^1(\G; \mathfrak{g})$ be the Maurer-Cartan form on $\G$.  Let $P$ be a connected, simply-connected manifold admitting a $1$-form $\mu \in \Omega^1(P; \mathfrak{g})$ satisfying $d\mu = -\mu \wedge \mu$.  Then there exists a map $F \colon P \to \G$, unique up to composition with left-translation in $\G$, such that $F^*\omega = \mu$.
  %Let $P$ be a connected, simply-connected smooth manifold, let $G$ be a Lie group with Lie algebra $\mathfrak{g}$, let $\omega \in \Omega^1(G; \mathfrak{g})$ denote the Maurer-Cartan form of $G$.  If $\mu \in \Omega^1(P; \mathfrak{g})$ is a $1$-form on $P$ satisfying $d\mu = -\mu \wedge \mu$, then there exists a smooth map $F \colon P \to G$, unique up to composition with left-translation, such that $F^*\omega = \mu$.
 \end{thm}
 
 \subsection{Associative stabilisers in $\SO(3)$}\label{ssect:assocstab}
 
 As the first step, we classify the subgroups of $\SO(3)$ that stabilise an associative plane. The classification of subgroups of $\SO(3)$ up to conjugacy is well-known: the subgroups are
 \begin{equation*}
 \mathrm{O}(2), \ \SO(2), \ \mathrm{A}_5, \ \mathrm{S}_4, \ \mathrm{A}_4, \ \mathrm{D}_n, \ \mathbb{Z}_n.
 \end{equation*}
 Let $\G$ be a group on this list. If $\G$ stabilises an associative 3-plane, then this 3-plane must be a three-dimensional subrepresentation of $\R^7,$ where $\R^7$ is viewed as a $\G$-representation by restriction of the $\SO(3)$ representation. For each group $\G$ on the list we will determine the fixed associative 3-planes by first finding all three dimensional $\G$-subrepresentations of $\R^7$ and checking which are associative.
 
 For our calculations, we identify the irreducible $\SO(3)$-representation $\R^7$ with the space $\mathcal{H}_3$ of harmonic cubic polynomials in three variables $x, y,$ and $z$. We use the following basis of $\mathcal{H}_3$: %The $\SO(3)$-action on $\R^7$ is induced by the standard action on $\R^3$.  
 \begin{align*}
 e_1 = \tfrac{1}{5} x & \left( 2 x^2 - 3y^2 - 3 z^2 \right) \\
 e_2 = \tfrac{\sqrt{6}}{10} z \left( 4x^2 - y^2 -z ^2 \right), \:\:\: & e_3 = \tfrac{\sqrt{6}}{10} y \left( 4x^2 - y^2 -z ^2 \right) , \\
 e_4 = \tfrac{2\sqrt{15}}{5} x y z, \:\:\: & e_5 = \tfrac{\sqrt{15}}{5} x \left( y^2 -z^2 \right), \\
 e_6 = \tfrac{\sqrt{10}}{10} z \left( 3 y^2 - z^2 \right), \:\:\: & e_7 = \tfrac{\sqrt{10}}{10} y \left( y^2 - 3 z^2 \right),
 \end{align*}
 which agrees with our conventions for $\omega_1, \ldots, \omega_7$ (\ref{eq:MCSO5om}).

 \subsubsection{Continuous stabiliser}\label{sssect:ContStab}
 
 Let $\G$ be a continuous subgroup of $\SO(3).$ The identity component of $\G$ is a closed 1-dimensional subgroup, and hence is conjugate to the group $\SO(2) \subset \SO(3)$ consisting of the rotations about the $x$-axis. Let $R_{\alpha} \in \SO(2)$ be the element representing rotation of an angle $\alpha$ about the $x$-axis. Then $R_{\alpha}$ fixes $\text{span}(e_1),$ acts as rotation by $\alpha$ on $\text{span}(e_2,e_3),$ acts as rotation by $2\alpha$ on $\text{span}(e_4,e_5),$ and acts as rotation by $3\alpha$ on $\text{span}(e_6, e_7).$ Thus, there are exactly three $3$-dimensional subrepresentations of $\R^7$ under this action of $\SO(2)$:
 \begin{align*}
 A_{123} = \text{span}(e_1, e_2, e_3), & & A_{145} = \text{span}(e_1, e_4, e_5), & & A_{167} = \text{span}(e_1, e_7, e_6).
 \end{align*}
 It is easy to check that each of these spaces are preserved by the action of $\O(2)$, and are in fact associative 3-planes.
 
 \begin{prop}
 	The $\SO(3)$-stabiliser of an associative 3-plane $E \in \mathrm{Gr}_{\mathrm{ass}}\! \left( \Hc_3 \right)$ is isomorphic to a continuous subgroup of $\SO(3)$ if and only if $E$ lies in the $\SO(3)$-orbit of one of the following associative 3-planes.
 	\begin{enumerate}
 		\item The associative 3-plane $A_{123}$ with stabiliser $\mathrm{O}(2).$ 
 		\item The associative 3-plane $A_{145}$ with stabiliser $\mathrm{O}(2).$ 
 		\item The associative 3-plane $A_{167}$ with stabiliser $\mathrm{O}(2).$
 	\end{enumerate}
 \end{prop}
 
 \begin{definition}
 	Let $\mathcal{O}_{123}, \mathcal{O}_{145},$ and $\mathcal{O}_{167}$ denote the $\SO(3)$-orbits of $\text{span}(e_1, e_2, e_3), \text{span}(e_1, e_4, e_5),$ and $\text{span}(e_1, e_6, e_7)$ respectively in the Grassmannian of 3-planes in $\R^7.$
 \end{definition}
 
 \subsubsection{Icosahedral stabiliser}\label{sssect:IcoStab}
 
 Let $\ico \cong A_5$ denote the subgroup of $\SO(3)$ generated by
 \begin{align*}
 \left[ \begin{array}{ccc}
 -1 & 0 & 0 \\
 0 & -1 & 0 \\
 0 & 0 & 1
 \end{array} \right] , \left[ \begin{array}{ccc}
 0 & 0 & 1 \\
 1 & 0 & 0 \\
 0 & 1 & 0
 \end{array} \right] , \frac{1}{2} \left[ \begin{array}{ccc}
 1 & -\tau & \tfrac{1}{\tau} \\
 \tau & \tfrac{1}{\tau} & -1 \\
 \tfrac{1}{\tau} & 1 & \tau
 \end{array} \right],
 \end{align*}
 where $\tau = \tfrac{1+\sqrt{5}}{2}.$ The group $\ico$ is the symmetry group of the icosahedron with vertices
 \begin{equation}\label{eq:icos}
 (0, \pm \tau, \pm 1), \:\:\: (\pm 1, 0, \pm \tau), \:\:\: (\pm \tau, \pm 1, 0).
 \end{equation}
  A computation yields that the irreducible decomposition of $\R^7$ under the action of $\ico$ is given by
 \begin{align*}
 \R^7 = A_{\ico} \oplus C,
 \end{align*}
 where
 \begin{align*}
 A_{\ico} &= \text{span} \left(e_1 + \sqrt{3} e_5, \sqrt{3} \left( -1 + \sqrt{5} \right) e_2 - \left( 3 + \sqrt{5} \right) e_6, \sqrt{3} \left( 1 + \sqrt{5} \right) e_3 - \left( -3 + \sqrt{5} \right) e_7  \right), \\
 C &= \text{span} \left( e_4, \sqrt{3} e_1 - e_5, \left( 3 + \sqrt{5} \right) e_2 + \sqrt{3} \left( -1 + \sqrt{5} \right) e_6, \left( -3 + \sqrt{5} \right) e_3 + \sqrt{3} \left( 1 + \sqrt{5} \right) e_7 \right).
 \end{align*}
 The 3-plane $A_{\ico}$ is associative. We denote the $\SO(3)$-orbit of $A_{\ico}$ in the Grassmannian of 3-planes by $\mathcal{O}_{\ico}.$ Note that $C$ is equal to the space of cubics vanishing on the vertices (\ref{eq:icos}) of the icosahedron \cite{Hit09}.
 
 \subsubsection{Octahedral stabiliser}
 
 Let $\oct \cong S_4$ denote the subgroup of $\SO(3)$ generated by
 \begin{align*}
 \left[ \begin{array}{ccc}
 0 & -1 & 0 \\
 1 & 0 & 0 \\
 0 & 0 & 1
 \end{array} \right] , \left[ \begin{array}{ccc}
 0 & 0 & 1 \\
 1 & 0 & 0 \\
 0 & 1 & 0
 \end{array} \right] , \left[ \begin{array}{ccc}
 -1 & 0 & 0 \\
 0 & 0 & 1 \\
 0 & 1 & 0
 \end{array} \right],
 \end{align*}
 the symmetry group of the octahedron with vertices
 \begin{equation}
 (\pm 1, 0, 0), \:\:\: (0, \pm 1, 0), \:\:\: (0, 0, \pm 1).
 \end{equation}
 A computation yields that the irreducible decomposition of $\R^7$ under the action of $\oct$ is given by
 \begin{align}\label{eq:R7OctDecomp}
 \R^7 = A_\oct \oplus \text{span}\left( e_4 \right) \oplus W,
 \end{align}
 where
 \begin{align*}
 A_{\oct} = \text{span} \left( e_1, \sqrt{3} e_2 + \sqrt{5} e_6, \sqrt{3} e_3 - \sqrt{5} e_7 \right), \\
 W = \text{span} \left( e_5, \sqrt{5}e_2 - \sqrt{3} e_6, \sqrt{5} e_3 + \sqrt{3} e_7 \right). 
 \end{align*}
 The representations $A_{\oct}$ and $W$ of $\oct$ are not isomorphic, so a three dimensional subrepresentation of $\R^7$ is either $A_{\oct}$ or $W.$ The 3-plane $A_{\oct}$ is associative, while the 3-plane $W$ is not. We denote the $\SO(3)$-orbit of $A_{\oct}$ in the Grassmannian of 3-planes by $\mathcal{O}_{\oct}.$ Note that $A_{\oct}$ is the 3-plane spanned by the harmonic parts of the perfect cubes $x^3, y^3, z^3.$
 
 \subsubsection{Tetrahedral stabiliser}
 
 The tetrahedral group $\tet \cong A_4$ is the subgroup of $\SO(3)$ generated by 
 \begin{align*}
 \left[ \begin{array}{ccc}
 -1 & 0 & 0 \\
 0 & -1 & 0 \\
 0 & 0 & 1
 \end{array} \right] , \left[ \begin{array}{ccc}
 0 & 0 & 1 \\
 1 & 0 & 0 \\
 0 & 1 & 0
 \end{array} \right].
 \end{align*}
 Since $\tet$ is a subgroup of both $\ico$ and $\oct,$ the associative 3-planes $A_{\ico}$ and $A_{\oct}$ are $\tet$-invariant. The irreducible decomposition of $\R^7$ under $\tet$ is the same as the decomposition (\ref{eq:R7OctDecomp}) under $\oct.$ The spaces $A_{\oct}$ and $W$ are isomorphic as $\tet$ representations, and thus any 3-plane invariant under $\tet$ is of the form
 \begin{align*}
 P_{\theta} = \text{span} & \left( \cos \theta e_1 + \sin \theta e_5 , \cos \theta \left( \sqrt{3} e_2 + \sqrt{5} e_6 \right) + \sin \theta \left( \sqrt{5}e_2 - \sqrt{3} e_6 \right), \right. \\
  &  \left. \cos \theta \left( \sqrt{3} e_3 - \sqrt{5} e_7 \right) + \sin \theta \left( \sqrt{5} e_3 + \sqrt{3} e_7 \right) \right).
 \end{align*}
 Calculation shows that the 3-plane $P_\theta$ is associative if and only if $\theta \in \lbrace 0, 2\pi/3, 4\pi/3 \rbrace.$ When $\theta = 0,$ $P_0=A_{\oct}$ and when $\theta = 2\pi / 3, 4 \pi/3$ we find that $P_{\theta}$ lies in the orbit $\mathcal{O}_{\ico}.$ Thus, there are no associative 3-planes in $\mathcal{H}_3$ with stabiliser exactly equal to $\tet.$
 
 \begin{prop}
 	The $\SO(3)$-stabiliser of an associative 3-plane $E \in \mathrm{Gr}_{\mathrm{ass}} \left( \Hc_3 \right)$ is isomorphic to an irreducibly acting subgroup of $\SO(3)$ if and only if $E$ lies in the $\SO(3)$-orbit of one of the following associative 3-planes.
 	\begin{enumerate}
 		\item The associative 3-plane $A_{\ico}$ with stabiliser $\ico.$ 
 		\item The associative 3-plane $A_{\oct}$ with stabiliser $\oct.$
 	\end{enumerate}
 \end{prop}
 
 \subsubsection{Dihedral and cyclic stabiliser}
 
 Every cyclic subgroup of $\SO(3)$ of order $n$ is conjugate to the group $\mathbb{Z}_n \subset \SO(2) $ generated by the element $R_{2\pi /n}$ representing rotation of $2\pi /n$ about the $x$-axis. Under the irreducible representation $\SO(3) \to \SO(7)$, $R_{2\pi /n}$ fixes $e_1,$ acts as rotation by $2\pi / n$ on $\text{span}(e_2,e_3),$ acts as rotation by $4 \pi / n$ on $\text{span}(e_4,e_5),$ and acts as rotation by $6 \pi / n$ on $\text{span}(e_6, e_7).$ The dihedral group $\mathrm{D}_n \subset \SO(3)$ is generated by the elements of $\mathbb{Z}_n$ together with the reflection $C = \mathrm{diag} \left(-1, -1, 1 \right).$
 
 We restrict attention to the cases $n > 2,$ as the description of the associative planes stabilised by $\mathbb{Z}_2$ and $\mathrm{D}_2$ is more complicated than in the other cases. 
 
 If $n > 6,$ the irreducible decomposition under the action of $\mathbb{Z}_n$ is
 \begin{equation}\label{eq:CycIrrDeco6}
 \R^7 = \R e_1 \oplus \textrm{span} \left( e_2, e_3 \right) \oplus \textrm{span} \left( e_4, e_5 \right) \oplus \textrm{span} \left( e_6, e_7 \right),
 \end{equation}
 with the summands mutually non-isomorphic. It follows that any 3-plane fixed by the action of $\mathbb{Z}_n$ for $n > 6$ is equal to one of the associative 3-planes fixed by $\mathrm{O}(2)$ described in \S\ref{sssect:ContStab}.
 
 If $n=6,$ the summand $\textrm{span}(e_6, e_7)$ in the decomposition (\ref{eq:CycIrrDeco6}) reduces further to $\R e_6 \oplus \R e_7,$ while the other summands remain irreducible and are mutually non-isomorphic. The generator $R_{\pi/3}$ acts trivially on $\R \cdot e_1$ and acts by $-1$ on $\R e_6$ and $\R e_7.$ Thus, a 3-plane fixed by $\mathbb{Z}_6$ must be equal to one of the associative 3-planes fixed by $\mathrm{O}(2)$ described in \S\ref{sssect:ContStab}, or one of the 3-planes
 \begin{equation*}
 \textrm{span}(e_2, e_3, e_6), \:\:\: \textrm{span}(e_2, e_3, e_7), \:\:\: \textrm{span}(e_4, e_5, e_6), \:\:\: \textrm{span}(e_4, e_5, e_7).
 \end{equation*}
 Only the 3-planes fixed by $\mathrm{O}(2)$ are associative.
 
 If $n=5,$ the decomposition (\ref{eq:CycIrrDeco6}) is irreducible, but the summands $\textrm{span} \left( e_4, e_5 \right)$ and $\textrm{span} \left( e_6, e_7 \right)$ are isomorphic as $\mathbb{Z}_5$-modules via the map $(e_4, e_5) \mapsto (e_7, e_6).$ Thus, a 3-plane fixed by the action of $\mathbb{Z}_5$ must be equal to either $\textrm{span}(e_1, e_2, e_3)$ or of the form
 \begin{equation*}
 P^{(5)}( a,b ) = \mathrm{span} \left( e_1, \left( a_1 + i a_2 \right) \left( e_4 + i e_5 \right) + \left( b_1 + i b_2 \right) \left( e_6 - i e_7 \right) \right),
 \end{equation*}
 for complex numbers $a = a_1 + ia_2, b = b_1 + i b_2.$ The 3-planes $P(a,b)$ and $P(e^{2 \psi} a, e^{-3 \psi} b)$ are $\SO(3)$-equivalent, and simultaneously scaling $a$ and $b$ by a complex parameter does not change $P(a,b).$ The element $\mathrm{diag}(1, -1, -1) \in \SO(3)$ acts trivially on $\mathrm{span}(e_4, e_5),$ and by $-1$ on $\mathrm{span}(e_6, e_7).$ Thus, every $\mathbb{Z}_5$-invariant 3-plane is $\SO(3)$-equivalent to a plane of the form
 \begin{equation*}
 Q^{(5)}( {\theta} ) = \textrm{span}\left( e_1, \cos \theta e_4 + \sin \theta e_7, \cos \theta e_5 + \sin \theta e_6 \right),
 \end{equation*}
 for $0 \leq \theta < \pi.$ Calculation shows that all of these 3-planes are associative.
 
 If $n=4,$ the summand $\textrm{span}(e_4, e_5)$ in the decomposition (\ref{eq:CycIrrDeco6}) reduces further to $\R e_4 \oplus \R e_5.$ The other summands are irreducible. The generator $R_{\pi/2}$ acts trivially on $\R e_1$ and acts by $-1$ on $\R e_4$ and $\R e_5.$ The irreducible summands $\textrm{span}(e_2, e_3)$ and $\textrm{span}(e_6, e_7)$ are isomorphic as $\mathbb{Z}_4$-modules via the map $(e_2, e_3) \mapsto (e_7, e_6).$ Thus, by reasoning similar to the case $n = 5,$ a 3-plane fixed by the action of $\mathbb{Z}_4$ must be equal to $\textrm{span} (e_1, e_4, e_5)$ or lie in the $\SO(3)$ orbit of one of the 3-planes
 \begin{equation*}
 \begin{aligned}
 Q^{(4\textrm{a})} ( \theta ) &= \textrm{span} \left(e_1, \cos \theta \, e_2 + \sin \theta \, e_7, \cos \theta \,  e_3 + \sin \theta \, e_6 \right), \\
 Q^{(4\textrm{b})} \left( \psi, \theta \right) &= \textrm{span} \left(\cos \psi \, e_4 + \sin \psi \, e_5 , \cos \theta \, e_2 + \sin \theta \, e_7, \cos \theta \, e_3 + \sin \theta \, e_6 \right),
 \end{aligned}
 \end{equation*}
 for $0 \leq \psi, \theta < \pi.$ Calculation shows that all of the 3-planes $Q^{(4\textrm{a})} ( \theta )$ are associative, but none of the planes $Q^{(4\textrm{b})} \left( \psi, \theta \right)$ are.
 
 If $n=3,$ the summand $\textrm{span}(e_6, e_7)$ in the decomposition (\ref{eq:CycIrrDeco6}) reduces further to $\R e_6 \oplus \R e_7.$ The other summands are irreducible. The irreducible summands $\textrm{span}(e_2, e_3 )$ and $\textrm{span}(e_4, e_5)$ are isomorphic as $\mathbb{Z}_3$-modules via the map $\left( e_2, e_3 \right) \mapsto \left( e_5, e_4 \right)$. The generator $R_{2 \pi / 3}$ acts trivially on $\R e_1, \R e_6,$ and $\R e_7.$ By similar reasoning to the previous two cases, a 3-plane fixed by the action of $\mathbb{Z}_3$ must be equal to the associative 3-plane $\textrm{span}(e_1, e_7, e_6)$ with stabiliser $\mathrm{O}(2),$ or one of the planes
 \begin{equation*}
 Q^{(3)} ( \theta, a ) = \textrm{span} \left( a_1 e_1 + a_6 e_6 + a_7 e_7, \cos \theta e_2 + \sin \theta e_5, \cos \theta e_3 + \sin \theta e_4 \right),
 \end{equation*}
 here $a = (a_1, a_6, a_7)$ is an element of $S^2,$ and $0 \leq \theta < 2 \pi.$ Calculation shows that the 3-plane $Q^{(3)} ( \theta, a )$ is associative if and only if $(a_1, a_6, a_7)= ( \cos 2 \theta, \sin 2 \theta, 0).$

 \begin{prop}\label{prop:CycStab}
 	An associative 3-plane $E \in \mathrm{Gr}_{\mathrm{ass}} \left( \Hc_3 \right)$ is stabilised by a cyclic or dihedral subgroup of $\SO(3)$ with $n > 2$ if and only if $E$ lies in the $\SO(3)$-orbit of:
 	\begin{enumerate}
 		\item An associative 3-plane of the form
 		\begin{equation*}
 		 Q^{(5)} (\theta) = \mathrm{span}\left( e_1, \cos \theta e_4 + \sin \theta e_7, \cos \theta e_5 + \sin \theta e_6 \right),
 		\end{equation*}
 		with $\theta \in \left( 0, \pi \right) \setminus \left\lbrace \arccos \sqrt{3/5},  \pi /2 \right\rbrace,$ which has stabiliser $\mathbb{Z}_5.$
 		\item An associative 3-plane of the form 
 		\begin{equation*}
 		Q^{(4\textrm{a})} (\theta) = \mathrm{span} \left(e_1, \cos \theta e_2 + \sin \theta e_7, \cos \theta e_3 + \sin \theta e_6 \right),
 		\end{equation*}
 		with $\theta \in \left( 0, \pi \right) \setminus \left\lbrace \arccos \sqrt{6}/4, \pi /2 \right\rbrace,$ which has stabiliser $\mathbb{Z}_4.$
 		\item An associative 3-plane of the form
 		\begin{equation*}
 		Q^{(3)} (\theta) = \mathrm{span} \left( \cos 2 \theta e_1 + \sin 2 \theta e_6, \cos \theta e_2 + \sin \theta e_5, \cos \theta e_3 + \sin \theta e_4 \right),
 		\end{equation*}
 		with $\theta \in \left( 0, \pi \right) \setminus \left\lbrace \arccos \sqrt{6}/3, \pi /2 \right\rbrace,$ which has stabiliser $\mathbb{Z}_3.$
% 		\item An associative 3-plane of the form
% 		\begin{equation*}
% 		Q^{(2b)}_{\mathbf{a}, \mathbf{x}, \mathbf{y}} = \mathrm{span} \left( \mathbf{a}, \mathbf{x}, \mathbf{y} \right),
% 		\end{equation*}
% 		where $\mathbf{a}, \mathbf{x}, \mathbf{y}$ are as described above and satisfy equation (\ref{eq:axycond}), which is fixed by the action of $\mathbb{Z}_2.$
 	\end{enumerate}
 \end{prop}

\begin{proof}
	The bulk of the proof has been completed above. The only remaining task is to verify that the restrictions on $\theta$ in each case ensure that the $\SO(3)$-stabiliser of each member of the families is exactly equal to the claimed group. We leave this to the reader.
\end{proof}

\subsection{Continuous stabiliser}

In this section we classify the associative submanifolds of $\Ber$ whose tangent space has $\SO(3)$-stabiliser isomorphic to a continuous subgroup $\G.$ By the results of \S\ref{ssect:assocstab}, $\G = \mathrm{O}(2)$ are there are three cases to consider: one for each of the orbits $\mathcal{O}_{123}, \mathcal{O}_{145},$ and $\mathcal{O}_{167}.$

\begin{remark}
    The $\SO(3)$-module $\Hc_3$ endows $\R^7$ with a natural $\SO(3)$-structure, which may be viewed as the flat analogue of the $\SO(3)$-structure on $\Ber$ underlying the nearly parallel $\G_2$-structure $\varphi.$ Landsberg \cite{LandsbergMin} has studied minimal submanifolds of $\mathbb{R}^{2n+1} = \Hc_n$ with Gauss map taking values in an orbit of $\SO(3)$ with continuous stabiliser. The results in this section may be thought of as a non-flat analogue of Landsberg's work in the case $n=3.$ 
\end{remark}

\subsubsection{Case I}

Let $f: N \to \Ber$ be an associative 3-fold whose Gauss map lies in $\mathcal{O}_{123}$ for all $p \in N.$ Thus, we may adapt frames on $N$ so that $\omega_4 = \omega_{{5}} = \omega_{{6}} = \omega_{{7}} = 0.$ Let $\mathcal{P}_{123} (N)$ denote the $\mathrm{O}(2)$-subbundle of $f^* \SO(5)$ corresponding to this frame adaptation,
\begin{align*}
\mathcal{P}_{123} (N) = \left\lbrace \left( p , e \right) \in f^* \SO(5) \mid \omega_4 = \omega_{{5}} = \omega_{{6}} = \omega_{{7}} = 0 \right\rbrace.
\end{align*}

On $\mathcal{P}_{123} (N),$ the forms $\omega_1, \omega_2, \omega_3$ are semi-basic, while the form $\gamma_1$ is a connection form. The equations
\begin{align*}
d \left[ \begin{array}{c}
\omega_4 \\
\omega_5 \\
\omega_6 \\
\omega_7
\end{array} \right] = \left[ \begin{array}{c}
0 \\
0 \\
0 \\
0
\end{array} \right] = - \frac{\sqrt{10}}{2} \left[ \begin{array}{ccc}
0 & \gamma_3 & -\gamma_2 \\
0 & \gamma_2 & \gamma_3 \\
0 & 0 & 0 \\
0 & 0 & 0
\end{array} \right] \wedge \left[ \begin{array}{c}
\omega_{{1}} \\
\omega_2 \\
\omega_{{3}}
\end{array} \right]
\end{align*}
imply that there exist functions $a_2$ and $a_3$ on $\mathcal{P}_{123}(N)$ with
\begin{equation*}
\gamma_2 = \tfrac{1}{\sqrt{6}} \left( a_2 \omega_2 + a_3 \omega_3 \right), \:\:\: \gamma_3 = \tfrac{1}{\sqrt{6}} \left( a_3 \omega_2 - a_2 \omega_3 \right). 
\end{equation*}
The structure equations (\ref{eq:BerStruct1}) and (\ref{eq:BerStruct2}) restricted to $\mathcal{P}_{123}(N)$ imply
\begin{equation}\label{eq:123Cart1}
\begin{aligned}
d \omega_1 &=  \left( 2 a_3 - \tfrac{2}{3} \right) \omega_2 \wedge \omega_3, \\
d \omega_2 & =  - \gamma_1 \wedge \omega_3 - a_2 \omega_1 \wedge \omega_2 + \left( - a_3 + \tfrac{2}{3} \right) \omega_1 \wedge \omega_3, \\
d \omega_3 & = \gamma_1 \wedge \omega_2 + \left( a_3 - \tfrac{2}{3} \right) \omega_1 \wedge \omega_2  - a_2 \omega_1 \wedge \omega_3, \\
d \gamma_1 & = \left( \tfrac{1}{6} a_2^2 + \tfrac{1}{6} a_3^2 + \tfrac{4}{9} \right) \omega_2 \wedge \omega_3.
\end{aligned}
\end{equation}
The exterior derivatives of (\ref{eq:123Cart1}) together with the structure equations (\ref{eq:BerStruct2}) imply that there exist functions $b_2, b_3$ on $\mathcal{P}_{123}(N)$ such that
\begin{equation}\label{eq:123da}
\begin{aligned}
d a_2 & = \left(a_2^2 - a_3^2 + \tfrac{2}{3} a_3 + \tfrac{8}{3} \right) \omega_1 + b_2 \omega_2 + b_3 \omega_3, \\
d a_3 & = 2 a_2 \left( a_3 - \tfrac{1}{3} \right) \omega_1 + b_3 \omega_2 - b_2 \omega_3. 
\end{aligned}
\end{equation}
It follows that the functions $ a_2,  a_3$ descend to give well-defined functions on $N.$

The structure equations (\ref{eq:123Cart1}) imply that the differential system $\mathcal{R} = \left\langle \omega_2, \omega_3 \right\rangle$ on $\mathcal{P}_{123} (N)$ is a Frobenius system. Thus the 4-manifold $\mathcal{P}_{123}(N)$ is foliated by the integral surfaces of $\mathcal{R}.$ The equations $\omega_2 = \omega_3 = \omega_4 = \ldots = \omega_7 = 0$ hold on the integral surfaces on $\mathcal{R},$ so it follows that the integral surfaces of $\mathcal{R}$ are $\mathcal{D}$-surfaces (see \S\ref{sect:Ruled}), and their projections to $N$ are $\mathcal{C}$-curves. Thus, $N$ is foliated by $\mathcal{C}$-curves, i.e. it is a ruled associative in the sense of Definition \ref{defn:ruled}, and Theorem $\ref{thm:RuledAssoc}$ applies.

Let $S$ denote the $J$-holomorphic curve in $\gra$ corresponding to $N.$ By equations (\ref{eq:omegazeta}), the condition $\omega_4 = \ldots = \omega_7 = 0$ on $\mathcal{P}_{123}(N)$ implies that $\zeta_2 = \zeta_3 = 0$ on $\mathcal{G}(S).$ It follows from Proposition \ref{prop:JNK} and the discussion following it that $S$ is the Gauss lift of a superminimal surface in $S^4$ or is a fibre of the map $\gra \to \mathbb{CP}^3.$ This proves the following theorem.

\begin{thm}
	Let $N \to \Ber$ be an associative submanifold of the Berger space such that the Gauss map of $N$ has image contained in $\mathcal{O}_{123}.$ Then $N$ is a ruled associative submanifold which is locally the $\pi \circ \lambda^{-1}$-image of the Gauss lift to $\gra$ of a superminimal surface in $S^4$ or of a fibre of the map $\gra \to \mathbb{CP}^3.$
\end{thm}

We now classify the homogeneous examples of this type.

\begin{thm}
	Let $N \to \Ber$ be a homogeneous associative submanifold of the Berger space such that the Gauss map of $N$ has image contained in $\mathcal{O}_{123}.$ Then either:
	\begin{enumerate}
		\item $N$ is $\SO(5)$-equivalent to (an open subset of) the orbit of the Veronese surface $h^{-1} \cdot \Sigma_0$ under the action of the group $\SO(2) \times \SO(3)_{\mathrm{std}},$ where
		\begin{equation*}
		h = \begin{small}
		\left[ \begin {array}{ccccc} 0&0&1&0&0\\ 0&0&0&1&0
		\\ 0&0&0&0&1\\ 1&0&0&0&0
		\\ 0&1&0&0&0 \end {array} \right].
		\end{small}
		\end{equation*}
		\item $N$ is $\SO(5)$-equivalent to (an open subset of) the orbit of the standard Veronese surface $\Sigma_0$ under the action of the group $\mathrm{U}(2).$
	\end{enumerate} 
\end{thm}

\begin{proof}
	Since $N$ is homogeneous, the functions $a_2$ and $a_3$ are constant on $N$. It follows from equations (\ref{eq:123da}) that $a_2 = 0$ and $a_3$ is equal to either $-4/3$ or $2.$
	
	First, suppose $a_3 = -4/3.$ Let
	\begin{align*}
	\kappa & = \tfrac{2}{3}\omega_1 -2 \gamma_1, & \chi_2 & = \tfrac{5\sqrt{2}}{3}\,\omega_3, \\
	\chi_1 & = -\tfrac{4}{3} \omega_1 -\gamma_1, & \chi_3 & = \tfrac{5\sqrt{2}}{3}\,\omega_2.
	\end{align*}	
	Then the Maurer-Cartan form of $\SO(5)$ restricted to $\mathcal{P}_{123}(N)$ satisfies
	\begin{equation*}
	h^{-1} \, \mu |_{\mathcal{P}_{123}(N)} \, h = \begin{small} \left[ \begin {array}{ccccc} 0&\kappa&0&0&0\\ -
	\kappa&0&0&0&0\\ 0&0&0&\chi_{{3}}&-\chi_{{2}}
	\\ 0&0&-\chi_{{3}}&0&\chi_{{1}}\\ 0
	&0&\chi_{{2}}&-\chi_{{1}}&0\end {array} \right]. \end{small}
	\end{equation*}
	It follows from Cartan's Theorem \ref{thm:MaurerCartan} that $N$ is $\SO(5)$-equivalent to (an open subset of) the homogeneous submanifold described in part 1 of the theorem.
	
	Next, suppose $a_3 = 2.$ Let
	 \begin{align*}
	 \nu & = -\tfrac{1}{3}\omega_1 -\tfrac{3}{2} \gamma_1, & \xi_2 & = \tfrac{5\sqrt{6}}{9}\,\omega_2, \\
	 \xi_1 & = -\omega_1 + \tfrac{1}{2} \gamma_1, & \xi_3 & = -\tfrac{5\sqrt{6}}{9}\,\omega_3.
	 \end{align*}
	 Then the Maurer-Cartan form of $\SO(5)$ restricted to $\mathcal{P}_{123}(N)$ satisfies
	 \begin{equation*}
	 \mu |_{\mathcal{P}_{123}(N)} = \left[ \begin {array}{ccccc} 0&0&0&0&0\\ 0&0&\nu+
	 	\xi_{{1}}&\xi_{{2}}&\xi_{{3}}\\ 0&-\nu-\xi_{{1}}&0&-
	 	\xi_{{3}}&\xi_{{2}}\\ 0&-\xi_{{2}}&\xi_{{3}}&0&\nu-
	 	\xi_{{1}}\\ 0&-\xi_{{3}}&-\xi_{{2}}&-\nu+\xi_{{1}}&0
	 	\end {array} \right],
	 \end{equation*}
	 and it follows from Cartan's Theorem \ref{thm:MaurerCartan} that $N$ is $\SO(5)$-equivalent to (an open subset of) the homogeneous submanifold described in part 2 of the theorem.
\end{proof}

The topology of the example homogeneous under $\SO(2) \times \SO(3)_{\mathrm{std}}$ is given by the quotient $\SO(2) \times \SO(3)_{\mathrm{std}} / \mathrm{O}(2)_{1,2},$ which is diffeomorphic to $\SO(3) / \mathbb{Z}_2$. As a ruled associative, it corresponds via Theorem \ref{thm:RuledAssoc} to the tangent lift of a totally geodesic $S^2 \subset S^4.$

The topology of the example homogeneous under $\mathrm{U}(2)$ is given by the quotient $\mathrm{U}(2) / \mathrm{O}(2)_{1,2},$ which is diffeomorphic to $S^3.$ As a ruled associative it corresponds via Theorem \ref{thm:RuledAssoc} to a fibre of the map $\gra \to \mathbb{CP}^3.$

%Define a differential ideal $\mathcal{I}_{123}$ with independence condition $\Omega_{123}$ on $\SO(5)$ by
%\begin{align*}
%\mathcal{I}_{123} = \left\langle \omega_4, \omega_4, \omega_6, \omega_7 \right\rangle, \:\:\: \Omega_{123} = \omega_1 \wedge \omega_2 \wedge \omega_3 \wedge \gamma_1.
%\end{align*}
%
%\begin{prop}
%	The integral manifolds of $\left( \mathcal{I}_{123}, \Omega_{123} \right)$ are in bijection with $\Sg^1$-adapted coframe bundles of associative submanifolds of $\Ber$ whose Gauss map has image contained in $\mathcal{O}_{123}.$
%\end{prop}

%\begin{center}
%	\begin{tikzcd}
%		& & \SO(5) \arrow[dl] \arrow[dr] & & \\
%		\Sigma^2 \arrow[r, hook, "s. min."] & S^4 & & \Ber & N^3 \arrow[l, hook', "Lans. assoc."]
%	\end{tikzcd}
%\end{center}

\subsubsection{Case II}

Define a differential ideal $\mathcal{I}_{145}$ with independence condition $\Omega_{145}$ on $\SO(5)$ by
\begin{align*}
\mathcal{I}_{145} = \left\langle \omega_2, \omega_3, \omega_6, \omega_7 \right\rangle, \:\:\: \Omega_{145} = \omega_1 \wedge \omega_4 \wedge \omega_5 \wedge \gamma_1.
\end{align*}
By construction, the integral manifolds of $\left( \mathcal{I}_{145}, \Omega_{145} \right)$ are in bijection with $\O(2)$-adapted coframe bundles of associative submanifolds of $\Ber$ whose Gauss map has image contained in $\mathcal{O}_{145}.$

Using the structure equations of $\SO(5),$ we compute
\begin{align*}
d \left[ \begin{array}{c}
\omega_2 \\
\omega_3 \\
\omega_6 \\
\omega_7
\end{array} \right] \equiv - \sqrt{2} \left[ \begin{array}{ccc}
-\sqrt{3} \gamma_2 & -\tfrac{\sqrt{5}}{2} \gamma_3 & -\tfrac{\sqrt{5}}{2} \gamma_2 \\
\sqrt{3} \gamma_3 & \tfrac{\sqrt{5}}{2} \gamma_2 & -\tfrac{\sqrt{5}}{2} \gamma_3 \\
0 & \tfrac{\sqrt{3}}{2} \gamma_3 & -\tfrac{\sqrt{3}}{2} \gamma_2 \\
0 & \tfrac{\sqrt{3}}{2} \gamma_2 & \tfrac{\sqrt{3}}{2} \gamma_3
\end{array} \right] \wedge \left[ \begin{array}{c}
\omega_1 \\
\omega_4 \\
\omega_5
\end{array} \right] \:\:\: \textup{mod} \:\: I_{145}
\end{align*}
It follows that on an integral manifold of $\left( \mathcal{I}_{145}, \Omega_{145} \right)$ we have $\gamma_2 = \gamma_3 = 0.$

\begin{thm} \label{thm:145Case}
	Let $N^3 \to \Ber$ be an associative submanifold of the Berger space such that the Gauss map of $N$ has image contained in $\mathcal{O}_{145}.$ Then $N$ is $\SO(5)$-equivalent to (an open subset of) the orbit of the Veronese surface $h^{-1} \cdot \Sigma_0$ under the action of the group $ {\SO(2) \times \SO(3)_{\mathrm{std}}},$ where
	\begin{align*}
	h = \begin{small} \left[ \begin {array}{ccccc} 0&0&1&0&0\\ 1&0&0&0&0
	\\ 0&1&0&0&0\\ 0&0&0&1&0
	\\ 0&0&0&0&1\end {array} \right]
	 \end{small} \in \SO(5).
	\end{align*}
\end{thm}

\begin{proof}
Let $P$ be an integral manifold of $\left( \mathcal{I}_{145}, \Omega_{145} \right)$.  Let
\begin{align*}
\kappa & = \textstyle \frac{4}{3}\omega_1 + \gamma_1 & \chi_4 & = \textstyle\frac{2\sqrt{5}}{3}\,\omega_4 \\
\chi_1 & = \textstyle -\frac{2}{3}\omega_1 + 2\gamma_1 & \chi_5 & = \textstyle\frac{2\sqrt{5}}{3}\,\omega_5
\end{align*}
Then the Maurer-Cartan form $\mu$, restricted to $P$, satisfies
\begin{equation*}
h\mu|_{P}h^{-1} = \begin{small}
\begin{bmatrix}
0 & -\kappa & 0 & 0 & 0\\
\kappa & 0 & 0 & 0 & 0 \\
0 & 0 & 0 & -\chi_4 & \chi_5 \\
0 & 0 & \chi_4 & 0 & -\chi_1 \\
0 & 0 & -\chi_5 & \chi_1 & 0
\end{bmatrix}.
\end{small}
\end{equation*}
An application of Cartan's Theorem \ref{thm:MaurerCartan} concludes the proof.
\end{proof}

A calculation shows that the associative $3$-fold of Theorem \ref{thm:145Case} is totally geodesic and has Ricci curvature given by $\frac{2}{9}\!\left( \omega_1^2 + 9\omega_4^2 + 9\omega_5^2\right)$. Its topology is given by the quotient $\SO(2) \times \SO(3)_{\mathrm{std}} / \mathrm{O}(2)_{1,2},$ which is diffeomorphic to $\SO(3) / \mathbb{Z}_2.$ This associative is ruled, and corresponds under Theorem \ref{thm:RuledAssoc} to the normal lift of a totally geodesic $S^2 \subset S^4.$ This construction is generalised in \S\ref{ssect:D4Assoc}.

% The Ricci curvature of this example is $\tfrac{1}{18} \left( 17 \omega_1^2 + \omega_4^2 + \omega_5^2 \right).$

\subsubsection{Case III}

Define a differential ideal $\mathcal{I}_{167}$ with independence condition $\Omega_{167}$ on $\SO(5)$ by
\begin{align*}
\mathcal{I}_{167} = \left\langle \omega_2, \omega_3, \omega_4, \omega_5 \right\rangle, \:\:\: \Omega_{167} = \omega_1 \wedge \omega_6 \wedge \omega_7 \wedge \gamma_1.
\end{align*}
By construction, the integral manifolds of $\left( \mathcal{I}_{167}, \Omega_{167} \right)$ are in bijection with $\O(2)$-adapted coframe bundles of associative submanifolds of $\Ber$ whose Gauss map has image contained in $\mathcal{O}_{167}.$

Using the structure equations of $\SO(5),$ we compute
\begin{align*}
d \left[ \begin{array}{c}
\omega_2 \\
\omega_3 \\
\omega_4 \\
\omega_5
\end{array} \right] \equiv - \sqrt{6} \left[ \begin{array}{ccc}
- \gamma_2 & 0 & 0 \\
\gamma_3 & 0 & 0 \\
0 & -\tfrac{1}{2} \gamma_2 & -\tfrac{1}{2} \gamma_3 \\
0 & -\tfrac{1}{2} \gamma_3 & \tfrac{1}{2} \gamma_2
\end{array} \right] \wedge \left[ \begin{array}{c}
\omega_1 \\
\omega_6 \\
\omega_7
\end{array} \right] \:\:\: \textup{mod} \:\: I_{167}
\end{align*}
It follows that on an integral manifold of $\left( \mathcal{I}_{167}, \Omega_{167} \right)$ we have $\gamma_2 = \gamma_3 = 0.$

\begin{thm}\label{thm:167Case}
	Let $N^3 \to \Ber$ be an associative submanifold of the Berger space such that the Gauss map of $N$ has image contained in $\mathcal{O}_{167}$. Then $N$ is $\SO(5)$-equivalent to (an open subset of) the orbit of the Veronese surface $h^{-1} \cdot \Sigma_0$ under the action of the group $\mathrm{U}(2)$, where
	\begin{equation*}
	 h = \begin{small} \left[ \begin {array}{ccccc} -1&0&0&0&0\\ 0&1&0&0&0
	\\ 0&0&1&0&0\\ 0&0&0&-1&0
	\\ 0&0&0&0&1\end {array} \right] \end{small} \in \SO(5).
	\end{equation*}
\end{thm}

\begin{proof}
Let $P$ be an integral manifold of $\left( \mathcal{I}_{167}, \Omega_{167} \right)$.  Define
\begin{align*}
\kappa & = 2\omega_1 - \gamma_1 & \chi_6 & = \textstyle \frac{2\sqrt{10}}{3}\omega_6 \\
\chi_1 & = \textstyle \frac{2}{3}\omega_1 - \gamma_1 & \chi_7 & = \textstyle \frac{2\sqrt{10}}{3}\omega_7
\end{align*}
Then the Maurer-Cartan form restricted to $P$ satisfies
\begin{equation*}
h^{-1} \, \mu|_P \, h = \frac{1}{2} \begin{small} \begin{bmatrix}
0 & 0 & 0 & 0 & 0 \\
0 & 0 & -\kappa - \chi_1 & \chi_6 & \chi_7 \\
0 & \kappa + \chi_1 & 0 & -\chi_7 & \chi_6 \\
0 & -\chi_6 & -\chi_7 & 0 & \kappa - \chi_1 \\
0 & \chi_7 & -\chi_6 & -\kappa + \chi_1 & 0
\end{bmatrix} \end{small}
\end{equation*}
As above, we conclude via Cartan's Theorem \ref{thm:MaurerCartan}.
\end{proof}

A calculation shows that the associative $3$-fold of Theorem \ref{thm:167Case} is totally geodesic and has Ricci curvature given by $\frac{2}{9}\!\left( \omega_1^2 + 19\omega_6^2 + 19\omega_7^2\right)$. Its topology is given by the quotient $\mathrm{U}(2) / \mathrm{O}(2)_{1,2},$ which is diffeomorphic to $S^3.$ This associative is ruled and corresponds under Theorem \ref{thm:RuledAssoc} to the lift to $\gra$ of a fibre of the twistor map $\mathbb{CP}^3 \to S^4$ via Proposition \ref{prop:JNK}.
 
 \subsection{Irreducibly acting stabiliser}
 
 In this section, we classify the associative submanifolds of $\Ber$ whose tangent space has $\SO(3)$-stabiliser isomorphic to a subgroup $\G$ acting irreducibly on $\R^3.$ By the results of \S\ref{ssect:assocstab} there are two subcases: $\G = \ico$ and $\G=\oct.$
 
 \subsubsection{Icosahedral case}
 
 Define 1-forms $\kappa_1, \kappa_2, \kappa_3, \beta_1, \ldots, \beta_4$ on $\SO(5)$ by
 \begin{align*}
 \kappa_1 & = \tfrac{1}{2} \omega_1 + \tfrac{\sqrt{3}}{2} \omega_5, & \beta_1 & = \omega_4, \\
 \kappa_2 & = - \tfrac{\sqrt{6} \left( 1 + \sqrt{5} \right)}{8} \omega_3 + \tfrac{\sqrt{2} \left( -3 + \sqrt{5} \right)}{8} \omega_7, & \beta_2 & = \tfrac{\sqrt{3}}{2} \omega_1 - \tfrac{1}{2} \omega_5, \\
 \kappa_3 & = \tfrac{\sqrt{6} \left( -1 + \sqrt{5} \right)}{8} \omega_2 - \tfrac{\sqrt{2} \left( 3 + \sqrt{5} \right)}{8} \omega_6, & \beta_3 & = \tfrac{\sqrt{2} \left( 3 +\sqrt{5} \right)}{8} \omega_2 + \tfrac{\sqrt{6} \left(-1 + \sqrt{5} \right)}{8} \omega_6, \\
  & & \beta_4 & = \tfrac{\sqrt{2} \left( -3 +\sqrt{5} \right)}{8} \omega_3 + \tfrac{\sqrt{6} \left(1 + \sqrt{5} \right)}{8} \omega_7.
 \end{align*}
 Define a differential ideal $\mathcal{I}_{\ico}$ with independence condition $\Omega_{\ico}$ on $\SO(5)$ by
 \begin{align}
 \mathcal{I}_{\ico} = \left\langle \beta_1, \beta_2, \beta_3, \beta_4 \right\rangle, \:\:\:\: \Omega_{\ico} = \kappa_1 \wedge \kappa_2 \wedge \kappa_3.
 \end{align}
 By construction, the integral manifolds of $\left( \mathcal{I}_{\ico}, \Omega_{\ico} \right)$ are in bijection with $\ico$-adapted coframe bundles of associative submanifolds of $\Ber$ whose Gauss map has image contained in $\mathcal{O}_{\ico}.$

Using the structure equations of $\SO(5),$ we may compute
\begin{align*}
d \left[ \begin{array}{c}
\beta_1 \\
\beta_2 \\
\beta_3 \\
\beta_4
\end{array}
\right] \equiv \sqrt{3} \left[ \begin{array}{ccc}
\gamma_1 & \gamma_2 & \gamma_3 \\
0 & \tfrac{-1 + \sqrt{5}}{2} \gamma_2 & \tfrac{-1-\sqrt{5}}{2} \gamma_3 \\
\tfrac{1+\sqrt{5}}{2} \gamma_2 & \tfrac{-1+\sqrt{5}}{2} \gamma_1 & 0 \\
\tfrac{1- \sqrt{5}}{2} \gamma_3 & 0 & \tfrac{1+\sqrt{5}}{2} \gamma_1
\end{array} \right] \wedge \left[ \begin{array}{c}
\kappa_1 \\
\kappa_2 \\
\kappa_3
\end{array} \right] \:\:\:\: \text{mod} \:\: {I_{\ico}}.
\end{align*}
It follows that on an integral manifold of $\mathcal{I}_{\ico}$ we have $\gamma_1 = \gamma_2 = \gamma_3 = 0.$

\begin{thm}\label{thm:IcoEg}
	Let $N^3 \to \Ber$ be an associative submanifold of the Berger space such that for all $p \in N,$ $T_p N$ has $\SO(3)$-stabiliser conjugate to $\ico \subset \SO(3).$ Then $N$ is $\SO(5)$-equivalent to (an open subset of) the orbit of the Veronese surface $h^{-1} \cdot \Sigma_0$ under the action of the group ${\SO(3)},$ where
	\begin{align*}
	h = \begin{small} \left[ \begin {array}{ccccc} -\tfrac{1}{4} & 0 & 0 & -\tfrac{\sqrt{15}}{4} & 0 \\
	0 & 1 & 0 & 0 & 0 \\
	0 & 0 & 1 & 0 & 0 \\
	\tfrac{\sqrt{15}}{4} & 0 & 0 & -\tfrac{1}{4} & 0 \\
	0 & 0 & 0 & 0 & 1 \end {array} \right] \end{small} \in \SO(5).
	\end{align*}
\end{thm}

\begin{proof}
	Let $P$ be an integral manifold of $\left( \mathcal{I}_{\ico}, \Omega_{\ico} \right)$ projecting to $N.$ By the work above, we have $\beta_1 = \cdots = \beta_4 = \gamma_1 = \gamma_2 = \gamma_3 = 0$ on $P.$ Thus, the Maurer-Cartan form $\mu$ of $\SO(5)$ restricts to $P$ to read
	\begin{align*}
	\mu|_P = \begin{small} \left[ \begin{array}{ccccc}
	0 & -\tfrac{\sqrt{3} \left( -1 + \sqrt{5} \right)}{6} \kappa_3 & -\tfrac{\sqrt{3} \left( 1 + \sqrt{5} \right)}{6} \kappa_2 & 0 & \tfrac{\sqrt{15}}{3} \kappa_1 \\
	\tfrac{\sqrt{3} \left( -1 + \sqrt{5} \right)}{6} \kappa_3 & 0 & -\tfrac{2}{3} \kappa_1 & \tfrac{1+3\sqrt{5}}{2} \kappa_3 & \tfrac{2}{3} \kappa_2 \\
	\tfrac{\sqrt{3} \left( 1 + \sqrt{5} \right)}{6} \kappa_2 & \tfrac{2}{3} \kappa_1 & 0 & \tfrac{1-3\sqrt{5}}{2} \kappa_2 & -\tfrac{2}{3} \kappa_3 \\
	0 & -\tfrac{1+3\sqrt{5}}{2} \kappa_3 & \tfrac{-1+3\sqrt{5}}{2} \kappa_2 & 0 & \tfrac{1}{3} \kappa_1 \\
	-\tfrac{\sqrt{15}}{3} \kappa_1 & -\tfrac{2}{3} \kappa_2 & \tfrac{2}{3} \kappa_3 & -\tfrac{1}{3} \kappa_1 & 0
	\end{array}
	\right], \end{small}
	\end{align*}
	and the structure equations of $\SO(5)$ imply that, on $P,$
	\begin{align*}
	d \kappa_1 & = - \tfrac{2}{3} \kappa_2 \wedge \kappa_3, & d \kappa_2 & = - \tfrac{2}{3} \kappa_3 \wedge \kappa_1, & d \kappa_3 & = - \tfrac{2}{3} \kappa_1 \wedge \kappa_2.
	\end{align*}
	It follows that $N$ is (an open subset of) a homogeneous submanifold of $\Ber.$ To determine this submanifold explicitly, note that, for $h$ defined as above, we have
	\begin{align}
	h^{-1} \mu|_N h = \begin{small} \frac{2}{3} \left[ \begin {array}{ccccc} 
	 0 & -\sqrt {3} \kappa_{{3}} & \sqrt {3}\kappa_{{2}} & 0 & 0 \\
	 \sqrt {3} \kappa_{{3}} & 0 & -\kappa_{{1}} & -\kappa_{{3}} & \kappa_{{2}} \\
	 -\sqrt {3}\kappa_{{2}} & \kappa_{{1}} & 0 & -\kappa_{{2}} & -\kappa_{{3}} \\
	 0 & \kappa_{{3}} & \kappa_{{2}} & 0 & -2\kappa_{{1}} \\
	 0 & -\kappa_{{2}} & \kappa_{{3}} & 2\kappa_{{1}} & 0 
	 \end {array} \right], \end{small}
	\end{align}
	and the right-hand-side is exactly our standard presentation (\ref{eq:MCSO5}) of the Lie algebra $\mathfrak{so}(3) \subset \mathfrak{so}(5)$.  The result now follows from an application of Cartan's Theorem \ref{thm:MaurerCartan}.
\end{proof}

In order to determine the topology of the manifold described in Theorem \ref{thm:IcoEg}, it is necessary to determine the intersection of the groups $\SO(3)$ and $h \SO(3) h^{-1},$ where the element $h$ is as defined in the statement of the theorem. The group $\SO(3)$ preserves the standard Veronese surface $\Sigma_0$ in $S^4$ defined by the equation 
\begin{equation}
v_1 \left( v_1^2 + \tfrac{3}{2} v_2^2 + \tfrac{3}{2} v_3^2 - 3 v_4^2 - 3 v_5^2 \right) + \tfrac{3\sqrt{3}}{2} v_4 \left( v_2^2 - v_3^2 \right) + 3 \sqrt{3} v_2 v_3 v_5 = 1,
\end{equation}
while the group $h \SO(3) h^{-1}$ preserves the Veronese surface $h \cdot \Sigma_0 \subset S^4$ defined by the equation
\begin{equation}
\begin{aligned}
\tfrac{1}{16} v_1 \left( 11 v_1^2 + 9\sqrt{15} v_1 v_4 + \left( -6 + 18 \sqrt{5} \right) v_2^2 + \left( -6 - 18 \sqrt{5} \right) v_3^2 - 33 v_4^2 + 12 v_5^2 \right) \\
+ \tfrac{3\sqrt{3}}{4} v_4 \left( \tfrac{-1-\sqrt{5}}{2} v_2^2 + \tfrac{1-\sqrt{5}}{2} v_3^2 -\tfrac{1}{4}\sqrt{5} v_4^2 + \sqrt{5} v_5^2 \right) + 3 \sqrt{3} v_2 v_3 v_5 = 1.
\end{aligned}
\end{equation}
Note that the standard Veronese surface $\Sigma_0$ may be described as the image of the $\SO(3)$-equivariant map $\nu: S^2 \to S^4$ given by
\begin{equation}
\nu: \left( x, y, z\right) \mapsto \left( x^2 - \tfrac{1}{2} y^2 -\tfrac{1}{2} z^2, \sqrt{3} xy, \sqrt{3} xz, \tfrac{\sqrt{3}}{2} y^2- \tfrac{\sqrt{3}}{2} z^2, \sqrt{3} y z \right).
\end{equation}
The map $\nu$ is a double cover of $\Sigma_0.$

Any element of the intersection $\SO(3) \cap h \SO(3) h^{-1}$ must preserve the intersection $\Sigma_0 \cap h \cdot \Sigma_0.$ Computation yields that $\Sigma_0 \cap h \cdot \Sigma_0$ is given by the $\nu$-image of the 20 points
\begin{equation*}
(\pm \tfrac{1}{\sqrt{3}},\pm\tfrac{1}{\sqrt{3}},\pm\tfrac{1}{\sqrt{3}}), \:\:  ( 0, \pm \tau, \pm \tfrac{1}{\tau}), \:\: (\pm \tfrac{1}{\tau},0, \pm \tau), \:\: ( \pm \tau, \pm \tfrac{1}{\tau}, 0),
\end{equation*}
where, as in \S\ref{sssect:IcoStab}, we have $\tau = \tfrac{1+\sqrt{5}}{2},$ the golden ratio. These points form the vertices of a regular dodecahedron in $S^2.$ Thus, the intersection $\SO(3) \cap h \SO(3) h^{-1}$ is a subgroup of the group of symmetries of this dodecahedron. This group is generated by the matrices
\begin{small}
	\begin{align*}
	\left[ \begin{array}{ccc}
	-1 & 0 & 0 \\
	0 & -1 & 0 \\
	0 & 0 & 1
	\end{array} \right] , \left[ \begin{array}{ccc}
	0 & 0 & 1 \\
	1 & 0 & 0 \\
	0 & 1 & 0
	\end{array} \right] , \frac{1}{2} \left[ \begin{array}{ccc}
	1 & -\tfrac{1}{\tau} & -{\tau} \\
	-\tfrac{1}{\tau} & \tau & -1 \\
	{\tau} & 1 & -\tfrac{1}{\tau}
	\end{array} \right].
	\end{align*}
\end{small}
In fact, it is possible to verify by explicit calculation (using, for example, Euler angles) that $\SO(3) \cap h \SO(3) h^{-1}$ is equal to the symmetry group of the dodecahedron. Thus, the associative group orbit described in Theorem \ref{thm:IcoEg} is diffeomorphic to the Poincar\'e homology sphere $\SO(3)/\ico.$

The induced metric on the associative of Theorem \ref{thm:IcoEg} has constant curvature.

\subsubsection{Octahedral case}

Define 1-forms $\kappa_1, \kappa_2, \kappa_3, \beta_1, \chi_1, \chi_2, \chi_3$ on $\SO(5)$ by
\begin{align*}
\kappa_1 & = \omega_1, & \chi_1 & = \omega_5 , \\
\kappa_2 & = -\tfrac{\sqrt{6}}{4} \omega_3 + \tfrac{\sqrt{10}}{4} \omega_7, & \chi_2 & =  -\tfrac{\sqrt{10}}{4} \omega_3 - \tfrac{\sqrt{6}}{4} \omega_7,  \\
\kappa_3 & = -\tfrac{\sqrt{6}}{4} \omega_2 - \tfrac{\sqrt{10}}{4} \omega_6, & \chi_3 & = \tfrac{\sqrt{10}}{4} \omega_2 - \tfrac{\sqrt{6}}{4} \omega_6, \\
\beta_1 & = \omega_4. & &
\end{align*}
Define a differential ideal $\mathcal{I}_{\oct}$ with independence condition $\Omega_{\oct}$ on $\SO(5)$ by
\begin{align*}
\mathcal{I}_{\oct} = \left\langle \beta_1, \chi_1, \chi_2, \chi_3 \right\rangle, \:\:\: \Omega_{\oct} = \kappa_1 \wedge \kappa_2 \wedge \kappa_3.
\end{align*}
By construction, the integral manifolds of $\left( \mathcal{I}_{\oct}, \Omega_{\oct} \right)$ are in bijection with $\oct$-adapted coframe bundles of associative submanifolds of $\Ber$ whose Gauss maps have image contained in $\mathcal{O}_{\oct}.$

\begin{thm}\label{thm:OctEg}
	Let $N^3 \to \Ber$ be an associative submanifold of the Berger space such that for all $p \in N,$ $T_p N$ has $\SO(3)$-stabiliser conjugate to $\oct \subset \SO(3).$ Then either:
	\begin{enumerate}
		\item $N$ is $\SO(5)$-equivalent to (an open subset of) the orbit of the Veronese surface $h^{-1} \cdot \Sigma_0$ under the action of the group ${\SO(3)},$ where
			\begin{align*}
			h = \begin{small} \left[ \begin {array}{ccccc} -1 & 0 & 0 & 0 & 0 \\
			0 & 1 & 0 & 0 & 0 \\
			0 & 0 & 1 & 0 & 0 \\
			0 & 0 & 0 & -1 & 0 \\
			0 & 0 & 0 & 0 & 1 \end {array} \right] \end{small} \in \SO(5).
			\end{align*}
		\item $N$ is $\SO(5)$-equivalent to (an open subset of) the orbit of the standard Veronese-Bor\r{u}vka surface $k^{-1} \cdot \Sigma_0$ under the action of the group ${\SO(3)}_{\textup{std}},$ where
		\begin{align*}
		k =  \begin{small} \left[ \begin {array}{ccccc} -1&0&0&0&0\\ 0&0&0&0&1
		\\ 0&0&0&1&0\\ 0&1&0&0&0
		\\ 0&0&1&0&0\end {array} \right]
		 \end{small} \in \SO(5).
		\end{align*}	
	\end{enumerate}
\end{thm}
\begin{proof}
	The structure equations of $\SO(5)$ imply
	\begin{align*}
	d \left[ \begin{array}{c}
	\beta_1 \\
	\chi_1 \\
	\chi_2 \\
	\chi_3
	\end{array}
	\right]  \equiv  \frac{\sqrt{15}}{2} \left[ \begin{array}{ccc}
	0 & 0 & 0 \\
	0 & \gamma_3 & \gamma_2 \\
	\gamma_3 & 0 & \gamma_1 \\
	\gamma_2 & \gamma_1 & 0
	\end{array} \right] \wedge \left[ \begin{array}{c}
	\kappa_1 \\
	\kappa_2 \\
	\kappa_3
	\end{array} \right] \:\:\:\: \text{mod} \:\: {I_{\oct}}.
	\end{align*}
	It follows that on an an integral manifold $P$ of $\mathcal{I}_{\oct}$ we must have $\gamma_i = a \kappa_i, \: i = 1, 2, 3,$ for some function $a$ on $P.$ The structure equations (\ref{eq:BerStruct2}) restricted to $P$ then imply $6a^2+a-1 = 0,$ so that either $a = -1/2$ or $a=1/3.$
	
	Suppose first that $P$ is an integral manifold of $\mathcal{I}_{\oct}$ with $a = -1/2$ on $P.$ The Maurer-Cartan form $\mu$ of $\SO(5)$ restricted to $P$ satisfies
	\begin{small}
			\begin{align*}
		h^{-1} \mu |_{P} h = \frac{5}{6} \left[ \begin {array}{ccccc} 
		0 & -\sqrt {3} \kappa_{{3}} & \sqrt {3}\kappa_{{2}} & 0 & 0 \\
		\sqrt {3} \kappa_{{3}} & 0 & -\kappa_{{1}} & -\kappa_{{3}} & \kappa_{{2}} \\
		-\sqrt {3}\kappa_{{2}} & \kappa_{{1}} & 0 & -\kappa_{{2}} & -\kappa_{{3}} \\
		0 & \kappa_{{3}} & \kappa_{{2}} & 0 & -2\kappa_{{1}} \\
		0 & -\kappa_{{2}} & \kappa_{{3}} & 2\kappa_{{1}} & 0 
		\end {array} \right],
		\end{align*}
	\end{small}
	where $h \in \SO(5)$ is the element defined in the statement of the theorem.  It follows from Cartan's Theorem \ref{thm:MaurerCartan} that $N = \pi \left( P \right)$ is $\SO(5)$-equivalent to (an open subset) of the homogeneous submanifold described in part 1 of the theorem.
	
	Suppose now that $P$ is an integral manifold of $\mathcal{I}_{\oct}$ with $a = 1/3$ on $P.$ The Maurer-Cartan form $\mu$ of $\SO(5)$ restricted to $P$ satisfies
	\begin{align*}
	k^{-1} \mu |_{P} k = \frac{5}{3} \begin{small} \left[ \begin {array}{ccccc} 0&0&0&0&0\\ 0&0&0&0&0
	\\ 0&0&0&\kappa_{{3}}&-\kappa_{{2}}
	\\ 0&0&-\kappa_{{3}}&0&\kappa_{{1}}
	\\ 0&0&\kappa_{{2}}&-\kappa_{{1}}&0\end {array}
	\right], \end{small}
	\end{align*}
	where $k \in \SO(5)$ is the element defined in the statement of the theorem.  As above, it follows that $N = \pi \left( P \right)$ is $\SO(5)$-equivalent to (an open subset) of the homogeneous submanifold described in part 2 of the theorem.
\end{proof}

One can check that the intersection $\SO(3) \cap h \SO(3) h^{-1}$ is equal to the group $\oct.$ It follows that the topology of the first example is $\SO(3)/\oct.$ The intersection $\SO(3) \cap k \SO(3)_{\textup{std}} k^{-1}$ is equal to the dihedral group $\textrm{D}_2,$ and so the topology of the second example is $\SO(3)/\mathrm{D}_2.$ Both examples have constant curvature.

\subsection{Cyclic and dihedral stabiliser}

In this section we classify the associative submanifolds of $\Ber$ whose tangent space has $\SO(3)$-stabiliser isomorphic to a cyclic subgroup or dihedral subgroup $\G$ having an element of order greater than 2. By the results of \S\ref{ssect:assocstab}, there are three subcases to consider: $\G = \mathbb{Z}_5, \G = \mathbb{Z}_4,$ and $\G = \mathbb{Z}_3.$

%We do not consider here the case $\G = \mathbb{Z}_2$ due to fact that the space of associative 3-planes in $\Hc_3$ fixed by $\mathbb{Z}_2$ is of a relatively large dimension (see Proposition \ref{prop:CycStab}), which complicates the required calculations.

\subsubsection{The case $\G = \mathbb{Z}_5$}

Let $f: N \to \Ber$ be an associative 3-fold whose tangent space has $\SO(3)$-stabilier everywhere equal to $\mathbb{Z}_5.$ By Proposition \ref{prop:CycStab}, we may adapt frames on $N$ so that $\omega_2 = \omega_3 = 0$ and 
\begin{equation}\label{eq:Z5vanish}
- \sin \theta \, \omega_4 + \cos \theta \, \omega_7 = 0, \:\:\: - \sin \theta \, \omega_5 + \cos \theta \, \omega_6 = 0,
\end{equation}
for some function $\theta.$ Let $\mathcal{P}_5 (N)$ denote the $\mathbb{Z}_5$-subbundle of $f^* \SO(5)$ corresponding to this frame adaptation. The function $\theta$ is well-defined on $\mathcal{P}_5(N).$ Define 1-forms $\kappa_1, \kappa_2, \kappa_3$ on $\mathcal{P}_5(N)$ by
\begin{equation}\label{eq:Z5kapdefn}
\begin{aligned}
\kappa_1 &= \omega_1, \\
\kappa_2 &= \cos \theta \, \omega_4 + \sin \theta \, \omega_7, \\
\kappa_3 &= \cos \theta \, \omega_5 + \sin \theta \, \omega_6.
\end{aligned}
\end{equation}
These forms are a basis for the semibasic forms on $\mathcal{P}_5(N),$ and the induced metric on $N$ pulls back to $\mathcal{P}_5(N)$ as
\begin{equation*}
g_N = \kappa_1^2 + \kappa_2^2 + \kappa_3^2.
\end{equation*}

On $\mathcal{P}_5(N),$ the exterior derivatives of the equations $\omega_2 = \omega_3 = 0$ and equations (\ref{eq:Z5vanish}) imply that there exist functions $t_2, t_3$ on $\mathcal{P}_{5}(N)$ such that
\begin{equation}\label{eq:Z5conseq}
\begin{aligned}
d \theta &= t_2 \kappa_2 + t_3 \kappa_3, & \gamma_2 &= 0, \\
\gamma_1 &= \tfrac{1}{5 \sin \theta \cos \theta} \left(-t_3 \kappa_2 + t_2 \kappa_3 \right), & \gamma_3 &= 0. \\
\end{aligned}
\end{equation}
Substituting these equations in to the structure equation (\ref{eq:BerStruct1}) yields the equations
\begin{equation}\label{eq:Z5AssocStruct}
\begin{aligned}
d \kappa_1 &= - \tfrac{2}{3} \kappa_2 \wedge \kappa_3, \\
d \kappa_2 &= -\tfrac{2}{3} \kappa_3 \wedge \kappa_1 + \tfrac{5 \cos^2 \theta - 3}{5 \sin \theta \cos \theta} t_3 \kappa_2 \wedge \kappa_3, \\
d \kappa_3 &= -\tfrac{2}{3} \kappa_1 \wedge \kappa_2 - \tfrac{5 \cos^2 \theta - 3}{5 \sin \theta \cos \theta} t_2 \kappa_2 \wedge \kappa_3.
\end{aligned}
\end{equation}

The structure equations (\ref{eq:Z5AssocStruct}) imply that the differential system $\mathcal{R} = \langle \kappa_2, \kappa_3 \rangle$ on $\mathcal{P}_5(N)$ is a Frobenius system. Thus, the 3-manifold $\mathcal{P}_5(N)$ is foliated by the integral curves of $\mathcal{R}.$ In fact, from (\ref{eq:Z5kapdefn}) and the definition of $\mathcal{P}_5(N),$ the integral curves of $\mathcal{R}$ project to $N$ to be $\mathcal{C}$-curves in $\Ber.$ Thus, $N$ is foliated by $\mathcal{C}$-curves, i.e. it is a ruled associative in the sense of Definition \ref{defn:ruled}.

\begin{prop}
	Let $N \to \Ber$ be an associative submanifold of the Berger space such that the tangent space of $N$ has $\SO(3)$-stabiliser everywhere equal to $\mathbb{Z}_5.$ Then $N$ is a ruled associative submanifold which is locally the $\pi \circ \lambda^{-1}$ image of a $J$-holomorphic curve in $\gra$ satisfying $\zeta_1 = \zeta_4 = 0.$ Conversely, any $J$-holomorphic curve in $\gra$ satisfying $\zeta_1 = \zeta_4 = 0$ gives rise to such an associative submanifold via Theorem \ref{thm:RuledAssoc}. 
\end{prop}

\begin{proof}
	The fact that $N$ is ruled has been demonstrated above. Theorem \ref{thm:RuledAssoc} guarantees that $N$ is locally the $\pi \circ \lambda^{-1}$ image of a $J$-holomorphic curve in $\gra,$ and the conditions $\zeta_1 = \zeta_4$ follow from the frame adaptation $\omega_2 = \omega_3 = 0$ and the equations $\gamma_2 = \gamma_3 = 0$ (\ref{eq:Z5conseq}), together with equation (\ref{eq:omegazeta}).
	
	To prove the converse, note that on a $J$-holomorphic curve $S$ with $\zeta_1 = \zeta_4 = 0$ we have $\left\lvert W_2 \right\rvert^2 + \left\lvert W_3 \right\rvert^2 = 1,$ and we may adapt frames so that
	\begin{equation}
	W_2 = \cos \theta, \:\:\: W_3 = -i \sin \theta.
	\end{equation}
	It follows that the tangent spaces of the corresponding ruled associative have $\SO(3)$-stabiliser everywhere isomorphic to $\mathbb{Z}_5.$
\end{proof}

\begin{thm}
	Let $S$ be a $J$-holomorphic curve in $\gra$ with $\zeta_1 = \zeta_4 = 0.$ Then $S$ is locally the normal lift of a superminimal surface in $S^4.$ Conversely, the normal lift of a superminimal surface in $S^4$ gives a $J$-holomorphic curve in $\gra$ satisfying $\zeta_1 = \zeta_4 = 0.$
\end{thm}

\begin{proof}
	Let $S$ be a holomorphic curve in $\gra$ with $\zeta_1 = \zeta_4 = 0.$ Using the notation of \S\ref{sect:Ruled}, the Maurer-Cartan form of $\SO(5)$ restricted to $\mathcal{G}(S)$ is
	\begin{equation}\label{eq:Z5SO5MC}
	\left[ \begin{array}{ccccc}
	0 & 0 & 0 & \sqrt{2} \Re \zeta_2 & - \sqrt{2} \Im \zeta_2 \\
	0 & 0 & \rho_1 - \rho_2 & \Re \zeta_3 & \Im \zeta_3 \\
	0 & -\rho_1 + \rho_2 & 0 & \Im \zeta_3 & - \Re \zeta_3 \\
	- \sqrt{2} \Re \zeta_2 & - \Re \zeta_3 & - \Im \zeta_3 & 0 & \rho_1 + \rho_2 \\
	\sqrt{2} \Im \zeta_2 & - \Im \zeta_3 & \Re \zeta_3 & - \rho_1 - \rho_2 & 0 
	\end{array} \right].
	\end{equation}
	
	Now let $h: U \to S^4$ be a superminimal surface in $S^4,$ and define an adapted coframe bundle $\mathcal{A}(U) \subset h^* \SO(5)$ by letting the fibre of $\mathcal{A}(U)$ at $p \in U$ be
	\begin{equation*}
	\mathcal{A}(U)_p = \left\lbrace \left( \mathbf{e}_1, \ldots, \mathbf{e}_5 \right) \in \SO(5) \mid \mathbf{e}_1 = p, \:\: T_p U = \mathrm{span} \left(\mathbf{e}_4, \mathbf{e}_5 \right) \right\rbrace.
	\end{equation*}
	The nature of the frame adaptation and the superminimality condition on $U$ together imply that the pullback of Maurer-Cartan form of $\SO(5)$ to $\mathcal{A}(U)$ may be written (with a judicious choice of notation) identically to (\ref{eq:Z5SO5MC}). The result then follows from Cartan's Theorem \ref{thm:MaurerCartan}.
\end{proof}

\subsubsection{The case $\G = \mathbb{Z}_4$}\label{ssect:D4Assoc}

Reasoning similar to the previous section implies that associative 3-folds whose tangent space has $\SO(3)$-stabiliser everywhere equal to $\mathbb{Z}_4$ are ruled associative submanifolds. The proof of the following proposition is then straightforward.

\begin{prop}
	Let $N \to \Ber$ be an associative submanifold of the Berger space such that the tangent space of $N$ has $\SO(3)$-stabiliser everywhere equal to $\mathbb{Z}_4.$ Then $N$ is a ruled associative submanifold which is locally the $\pi \circ \lambda^{-1}$ image of a $J$-holomorphic curve in $\gra$ satisfying $\zeta_2 = 0.$ Conversely, any $J$-holomorphic curve in $\gra$ satisfying $\zeta_2 = 0$ gives rise via Theorem \ref{thm:RuledAssoc} to an associative submanifold whose tangent space is fixed by $\mathbb{Z}_4$.
\end{prop}

By Proposition \ref{prop:JNK} and the surrounding discussion, the $J$-holomorphic curves in $\gra$ with $\zeta_2=0$ are exactly the lifts of $J_{\mathrm{NK}}$-holomorphic curves in $\mathbb{CP}^3.$

\subsubsection{The case $\G = \mathbb{Z}_3$}

By repeatedly applying the identity $d^2 = 0$ to the structure equations of an associative submanifold $N$ whose tangent spaces are everywhere fixed by the action of $\mathbb{Z}_3,$ it is possible to show that the tangent spaces of $N$ are in fact stabilised by a group strictly larger than $\mathbb{Z}_3.$ The details of this calculation are routine but long, so we do not include them.

\begin{prop}
	There are no associative submanifolds of $\Ber$ whose tangent spaces have $\SO(3)$-stabiliser everywhere exactly equal to $\mathbb{Z}_3.$
\end{prop}

\bibliography{BergerAssocRefs}

\Addresses

\end{document}